\documentclass[12pt]{extarticle}
\usepackage{amsmath, amsthm, amssymb}
\usepackage[colorlinks=true,linkcolor=blue,urlcolor=blue,citecolor=blue]{hyperref}
\usepackage{cleveref}
\usepackage{graphicx}
\usepackage{caption}
\usepackage{enumerate}
\usepackage[dvipsnames]{xcolor}
\usepackage{listings}
\usepackage{amsfonts}
\usepackage{maple}
\usepackage{verbatim}
\usepackage{tikz,tikz-cd,tikz-3dplot}
\usepackage{amssymb}
\usetikzlibrary{matrix}
\usetikzlibrary{arrows}
\usepackage{algorithm}
\usepackage{url}
\usepackage{moreverb}
\usepackage{sverb}
\usepackage{algpseudocode}
\usepackage{algorithmicx}
\usepackage{euscript}
\usepackage[T1]{fontenc}
\usepackage{breqn}
\usepackage{mathrsfs}
\usepackage{lmodern}
\usepackage{mathtools}
\usepackage{url}
\oddsidemargin \evensidemargin
\usepackage{makecell}
\usepackage{array}

\newtheorem{theorem}{Theorem}
\newtheorem{proposition}[theorem]{Proposition}

\newtheorem{corollary}[theorem]{Corollary}

\theoremstyle{definition}
\newtheorem{definition}[theorem]{Definition}

\newtheorem{remark}[theorem]{Remark}

\newtheorem{example}[theorem]{Example}

\newtheorem*{definition*}{Definition}
\numberwithin{theorem}{section}

\newcommand{\EE}{\mathbb{E}}
\newcommand{\RR}{\mathbb{R}}

\newcommand{\NN}{\mathbb{N}}
\newcommand{\KK}{\mathbb{K}}
\newcommand{\DD}{\mathbb{D}}
\newcommand{\Initp}{\mathcal{I}_{p}}
\newcommand{\Init}{\mathcal{I}}
\newcommand{\Dsigma}{\mathcal{D}_{\sigma}(\KK,x)}
\DeclareMathOperator{\ord}{ord}

\newcommand*{\QEDA}{\hfill\ensuremath{\blacksquare}}

\providecommand{\keywords}[1]
{
  \small	
  \textbf{\textit{Keywords---}} #1
}

\title{\bf Computing with D-Algebraic Sequences}

\author{Bertrand Teguia Tabuguia}
\date{}
\begin{document}
\maketitle

\begin{abstract}
  \noindent A sequence is difference algebraic (or D-algebraic) if finitely many shifts of its general term satisfy a polynomial relationship; that is, they are the coordinates of a generic point on an affine hypersurface. The corresponding equations are denoted algebraic difference equations (ADEs). We propose a formal definition of D-algebraicity for sequences and investigate algorithms for their closure properties. We show that subsequences of D-algebraic sequences, indexed by arithmetic progressions, satisfy ADEs of the same orders as the original sequences. Additionally, we discuss the special difference-algebraic nature of holonomic and $C^2$-finite sequences.
\end{abstract}

\keywords{Affine D-algebraic sequence, rationalizing difference polynomial, subsequence}

\section{Introduction}

Some well-known nonlinear recurrence relations arose with the Babylonian (or Heron's) method and Aitken extrapolation. The latter was introduced to estimate limits of convergent sequences. To a sequence of general term $s(n)$, the $\Delta^2$ process, or Aitken's transformation \cite{Aitken1927}, associates the term

\begin{equation}\label{eq:aitken}
 t(n) = s(n) - \frac{(\Delta s(n))^2}{\Delta^2 s(n)} = s(n) - \frac{\left(s(n+1)-s(n)\right)^2}{s(n+2)-2\,s(n+1)+s(n)},
\end{equation}
to accelerate the rate of convergence to $\lim_{n\to \infty} s(n)<\infty$. As discussed in \cite{weniger1989nonlinear}, this transformation has numerous applications in theoretical physics and other sciences \cite{baker1975essentials,brezinski1985convergence}. This paper presents a general study of a class of sequences that is stable under such transformations. These sequences satisfy nonlinear rational or polynomial recursions, similar to the one presented in \eqref{eq:aitken}. Although we do not aim to describe an algorithm for numerical computations, some of our examples are motivated by sequence acceleration in numerical analysis.

To give a glimpse of what our symbolic computations provide for Aitken's $\Delta^2$ process, consider approximating $\sqrt{\ell},\, \ell>0,$ using the Babylonian method. The iteration is governed by the recursion
\begin{equation}\label{eq:heron}
    s(n+1) = \frac{1}{2}\left(s(n)+\frac{\ell}{s(n)}\right),
\end{equation}
with the chosen initial term $s(0)=s_0$. Using our result, we can systematically show that the Aitken transformation yields the sequence with initial term $t(0)=\frac{\ell (\ell+3 s_0^2)}{s_0(s_0^2+3\ell)}$ and recursion
\begin{equation}
    t(n+1) = \frac{2\,\ell\,t(n)}{t(n)^2+\ell}=\frac{2\,\ell}{t(n)+\frac{\ell}{t(n)}}.
\end{equation}

\bigskip

\textbf{1.1. Context.} Following J. F. Ritt's algebraic approach to differential equations \cite{ritt1950differential}, the study of associated objects such as differential polynomials and differential varieties was further developed by Raudenbush \cite{raudenbush1934ideal}, Rosenfeld \cite{rosenfeld1959specializations}, and, more comprehensively with modern mathematics, by Kolchin \cite{kolchin1973differential}. For difference equations, which relate to the classical $\Delta$ operator defined as $\Delta f(x) \coloneqq f(x+1) - f(x)$, R. Cohn developed much of the theory \cite{cohndifferencealgebra}.

\smallskip

A key result of R. Cohn states that every nontrivial ordinary algebraically irreducible difference polynomial admits an abstract solution \cite{cohn1948manifolds}. This theorem has important implications for sequences over an algebraically closed field, say $\KK$, of characteristic zero. Indeed, specialized to the difference ring $\KK^{\NN}$ (with the monoid of natural integers $\NN = \{0, 1, \ldots\}$), Cohn's theorem ensures the existence of sequences that are zeros of ordinary difference polynomial. We call them difference-algebraic (or D-algebraic) sequences and denote their set by $\KK^{\aleph_0}$ (see \Cref{def:dalgdef}).

\smallskip

The scarcity of practical algorithms for arithmetic computations with D-algebraic sequences can largely be attributed to the presence of zero-divisors in sequence rings. It is only recently that a theory for the zeros of algebraic (or polynomial) difference equations (ADEs) within these rings has been proposed \cite{ovchinnikov2020effective}. This is a non-trivial undertaking, as many established results from difference and differential fields do not readily translate to sequence rings \cite{hrushovski2007neumann}. Wibmer introduced a notion of dimension for systems of ADEs with solutions in sequence rings, showing that this dimension need not be an integer \cite{wibmer2021dimension}. Decidability results for solving systems of difference equations in sequences are discussed in \cite{pogudin2020solving}. 

\bigskip

\textbf{1.2. Related work.} In the spirit of \cite{kauers2007algorithm}, we employ difference algebra as a convenient language to investigate the closure properties of D-algebraic sequences. We aim to provide a computational framework for D-algebraic sequences from which known answers related to D-algebraic functions can also be addressed.

\smallskip

D-algebraic functions \cite{RAB2024d,teguia2025arithmetic} generalize functions defined by polynomial differential equations, encompassing classes such as D-finite functions \cite{kauers2023d} and differentially definable functions \cite{jimenez2020some} (see also \cite{teguia2021representation} for related power series). Similarly, D-algebraic sequences generalize sequences defined by polynomial equations involving their shifts and the index variable. These sequences arise naturally in computer science, for example, in the context of cost-register automata \cite{AlurDDRY13} and polynomial automata \cite{benedikt2017polynomial} over a unary alphabet. One also finds D-algebraic sequences in modeling, as illustrated with the discrete May-Leonard model by Roeger and Allen \cite{Roeger01012004}.

\smallskip

A particularly important subclass of D-algebraic sequences consists of rational recursive sequences, which are the zeros of difference polynomials linear in their highest shifts (or order terms) \cite{clemente2023rational,stevic2004more}. We refer to these polynomials as \textit{rationalizing difference polynomials}.

Of course, for a given sequence $(s(n))_{n\in\NN}$, the relation
\begin{equation}\label{eq:rationalizing}
    s(n+r)\,\mathsf{D}\!\left(s(n),s(n+1),\ldots,s(n+r_1)\right) = \mathsf{N}\!\left(s(n),s(n+1),\ldots,s(n+r_2)\right),
\end{equation}
with polynomials $\mathsf{N}, \mathsf{D}$, and integers $r> r_1,r_2$, is ``meaningful'' only if the sequence $(s(n))_{n\in\NN}$ is not a zero of both the difference polynomials encoded by $\mathsf{N}$ and $\mathsf{D}$ simultaneously. This, however, does not affect the definition of a rationalizing difference polynomial; instead, it emphasizes the type of sequence solution. In this paper, we discuss a special type of so-called \textit{generic (sequence) solutions} \cite[P. 32]{levin2008difference}.

\bigskip

\textbf{1.3. Contributions.} We propose a formal definition for D-algebraic sequences. In the case of irreducible difference polynomials, this definition corresponds to special germs of their generic solutions. We then generalize this perspective to accommodate D-algebraic sequences with zeros of arbitrary difference polynomials.

It was shown in \cite{teguia2024rational} that any D-finite (holonomic or P-recursive) equation can be converted into a rationalizing ADE of order bounded by the sum of its holonomic degree and order. It turns out that the corresponding holonomic sequences are often generic zeros of the resulting ADE. We provide a corrected proof of this statement here, as an error was inadvertently introduced in the final version of \cite{teguia2024rational}.

Another notable subclass of D-algebraic sequences is that of $C^2$-finite sequences, introduced by Jiminez-Pastor, Nuspl, and Pillwein in \cite{jimenez2023extension}. We establish a result analogous to Theorem 2 from \cite{teguia2024rational} to give a constructive proof of their D-algebraicity. This proof is based on a new algorithm, \Cref{algo:Algo2}, which converts any $C^2$-finite equation into a rationalizing difference polynomial. Notably, our algorithm makes no explicit use of Gr\"obner bases or linear system solving. While we observe similarities with the resultant method from the differential case \cite{jimenez2020some}, the primary concern of that approach is not to derive an equation in which the highest order term appears linearly. 

Note that the differential analogs of these conversions are far more straightforward. Indeed, in the differential case, the algorithm for D-finite functions essentially involves eliminating the independent variable via the resultant method, followed by computing the first derivative of the resulting differential polynomial. Similarly, unlike $D^2$-finite equations, the conversion of $C^2$-finite equations into rational recursions appears to require different techniques.

\smallskip

Regarding arithmetic operations, while the computational theory for the differential case is relatively well-established (see, e.g., \cite{boulier1995representation,robertz2014formal,van2019computing,teguia2025arithmetic} and references therein), an effective theory for the arithmetic of (nonlinear) difference-algebraic sequences has received less attention. This may also explain the lack of results concerning D-algebraic subsequences, and thus motivates the present work.

This paper demonstrates that D-algebraic sequences behave well under field operations, provided that mild assumptions are made about their consecutive terms. For illustration purposes, consider the following algebraic difference equation satisfied by $(n!)_{n\in\NN}$:

\begin{equation}\label{eq:factorial}
    s\! \left(n+2\right)=\frac{s\! \left(n+1\right) \left(s\! \left(n\right)+s\! \left(n+1\right)\right)}{s\! \left(n\right)}.
\end{equation}
This equation may be obtained using the algorithm from \cite{teguia2024rational} with the minimal D-finite equation satisfied by $n!$ as input. We may use this recursion to construct an ADE satisfied by $u(n)\coloneqq \frac{n!}{n!+1}$, which is not D-finite. For $s(n)$ satisfying \eqref{eq:factorial} with $s(0)=s(1)=1$, we have that $u(n)=~\frac{s(n)}{s(n)+1}$, a rational expression in $s(n)$. Provided that
\begin{equation}\label{eq:condfacto1}
    s(n)s(n+1)s(n+2)(s(n)+1)(s(n+1)+1)(s(n+2)+1) \neq 0,
\end{equation}
\Cref{algo:Algo1} computes the following ADE for $u(n)$, which we write as a recursion to make the coefficient of the order term explicit:
\begin{equation}\label{eq:ratfacto}
    u\! \left(n+2\right)=\frac{u\! \left(n+1\right) \left(2 u\! \left(n+1\right) u\! \left(n\right)-u\! \left(n+1\right)-u\! \left(n\right)\right)}{u\! \left(n+1\right)^{2} u\! \left(n\right)-u\! \left(n+1\right)^{2}+u\! \left(n+1\right) u\! \left(n\right)-u\! \left(n\right)}.
\end{equation}
By considering generic zeros of difference polynomials, we neglect all sequence solutions of the ADE in \eqref{eq:ratfacto} that vanish the denominator on its right-hand side. The conditions made in \eqref{eq:condfacto1} arise from shifts of the denominators in \eqref{eq:factorial} and the relation between $s(n)$ and $u(n)$, which are used during the Gr\"obner bases computations. Again, this condition may be simplified for generic sequences satisfying \eqref{eq:factorial}. However, these conditions can be applied to practical computations involving actual sequences. For a specific $n_0\in\NN$, it is required to verify that all these conditions are satisfied before using \eqref{eq:ratfacto} to compute the subsequent terms of $(u(n))_{n\geq n_0}$. Since for $(n!)_{n\in\NN}$, the non vanishing of the denominator in \eqref{eq:ratfacto} (for $u(n)$) and the conditions in \eqref{eq:condfacto1} are always satisfied, all terms of $(u(n))_{n\in\NN}$ may be computed with \eqref{eq:ratfacto} and its first two initial values. In essence, our algorithmic computations with D-algebraic sequences rely on these kinds of generic assumptions, which are highly dependent on the initial terms of these sequences. The latter already relates to a longstanding decision problem for $C$-finite sequences \cite{ouaknine2012decision}. However, we will not discuss this aspect further, as our interest in this paper lies only in the symbolic constructions with difference polynomials.

\smallskip

In addition to field operations and completing the work in \cite{teguia2024rational}, we prove that D-algebraic sequences are closed under partial sums, partial products, radicals (one may assume complex values and appropriate choices for branches), and composition with arithmetic progressions. This last result can be extended to composition with broader classes of strictly increasing integer sequences. We are unaware of comparable results for such compositions, and we consider our method for computing equations for D-algebraic subsequences another key contribution.

\smallskip

Our method for all these operations is based on the decomposition-elimination-prolongation method developed by Ovchinnikov, Pogudin, and Vo \cite{ovchinnikov2016effective}, which was adapted to the difference algebra setting in \cite{ovchinnikov2020effective} (see also \cite{gao2009characteristic}). We construct difference polynomials for D-algebraic closure properties by adapting the result from \cite{teguia2025arithmetic}, which proposes a method to deal with difference polynomials that are nonlinear in their highest shifts.

\smallskip

The \texttt{NLDE} package \cite{teguia2023operations} contains a sub-package \texttt{DalgSeq} dedicated to the use of nonlinear algebra for difference equations. Most of the algorithms highlighted in this paper are implemented in \texttt{DalgSeq} available at \textcolor{blue}{\url{https://github.com/T3gu1a/D-algebraic-functions}}. 

\section{Definitions}
In this section, we fix some notations and establish a formal definition of a difference-algebraic sequence from the difference algebra setting. For further details on this theoretical perspective, we refer the reader to \cite{cohndifferencealgebra,levin2008difference}.
\begin{definition} A difference ring is a pair $(R,\Sigma)$ where $R$ is a commutative ring and $\Sigma$ is a finite set of pairwise commuting $R$-endomorphisms.  
\end{definition}
A difference ideal of $(R,\Sigma)$ is an ideal $J$ of $R$ such that $\sigma(J)\subset J$ for all $\sigma\in\Sigma$. The difference ring $(R,\Sigma)$ is also denoted $\Sigma$-ring $R$. When $\Sigma=\{\sigma\}$, the $\Sigma$-ring $R$ is simply denoted $\sigma$-ring and said to be ordinary. When $\left|\Sigma\right|>1$, $R$ is a partial difference ring. In this paper, we restrict ourselves to ordinary difference rings. We denote by $\NN$ the set of non-negative integers.

We primarily consider the $\sigma$-ring $(\KK, \sigma)$, where $\KK$ is an algebraically closed field of characteristic zero and $\sigma$ an endomorphism in $\KK$.
\begin{definition} The ring of difference polynomials in one difference indeterminate $x$ over $\KK$ , also called the free difference $\KK$-algebra in $x$, is the difference ring $\left(\KK[x_0,x_1,x_2,\ldots],\Tilde{\sigma}\right)$, where $\Tilde{\sigma}$ extends $\sigma$ as follows: $\Tilde{\sigma}(a)=\sigma(a)=a$, for any $a\in\KK$, and $\Tilde{\sigma}(x_j)=x_{j+1}$, $j\in\NN$. For simplicity, this difference ring is denoted by $\mathcal{D}_{\sigma}(\KK,x)\coloneqq \KK\left[\sigma^j(x) \mid j\in\NN \right]$ or $\KK\left[\sigma^{\infty}(x)\right]$ with the property $\sigma^{j+1}(x)=\sigma(\sigma^j(x)), j\in\NN$.
\end{definition}
\begin{definition}\label{def:diffpoly} The order of a difference polynomial $p\in\mathcal{D}_{\sigma}(\KK,x)$, denoted $\ord(p)$, is the maximum $j\in\NN$ such that $\sigma^j(x)$ appears in $p$. The degree of $p$ denoted $\deg(p)$ is the total degree of $p$ as a polynomial in the algebraic polynomial ring $\KK[x,\sigma(x),\ldots,\sigma^{\ord(p)}(x)]$. We say that $p$ is rationalizing when $p$ is linear in its highest shift term, i.e., the degree of $p$ with respect to $\sigma^{\ord(p)}(x)$ is $1$.
\end{definition}
 To any difference polynomial $p\in\mathcal{D}_{\sigma}(\KK,x)$, we associate a recurrence equation called {\it algebraic difference equation} obtained by the equality $p=0$. For a rationalizing difference polynomial, the associated recurrence equation is equivalent to a rational recursion.
\begin{example} The difference polynomial associated to \eqref{eq:factorial} is $\sigma^2(x)x-\sigma(x)(x+\sigma(x)),$ of order $2$ and degree $2$. For \eqref{eq:ratfacto}, the corresponding difference polynomial is of order $2$ and degree $4$. \QEDA 
\end{example}

In classical algebraic geometry, zeros of polynomials over $\KK$ correspond to points in $\KK^N$, the affine $N$-space over $\KK$, for some fixed $N\in\NN\setminus \{0\}$, defining algebraic sets or varieties with the ideals of those polynomials. In our setting, the corresponding points may be regarded with denumerable coordinates as we want any zero of $p\in\mathcal{D}_{\sigma}(\KK,x)$ to have coordinates that vanish $\sigma^j(p)$, for all $j\in\NN$. Every such point defines a sequence over $\KK$. We denote the set of such sequences over $\KK$ by $\KK^{\aleph_0}$, where $\aleph_0$ is {\it aleph zero}, the cardinality of $\NN$. This makes sense of the passage from a finite to an infinite (countable) number of coordinates. Note that $\KK^{\aleph_0} \subset \KK^{\NN}$, and the inclusion is strict since $(n^n)_{n\in\NN}$ is not a zero of a difference polynomial. The operations in $\KK^{\aleph_0}$ are inherited from the difference ring $\KK^{\NN}$ defined with the shift endomorphism.

\begin{definition}\label{def:defhomo} A homomorphism of difference rings $(R_1,\sigma_1)$ and $(R_2,\sigma_2)$ is a ring-homomorphism $\varphi\colon R_1\longrightarrow R_2$ such that $\varphi \circ \sigma_1 = \sigma_2 \circ \varphi$.
\end{definition}
\begin{definition}\label{def:defzero} Let $p\in\mathcal{D}_{\sigma}(\KK,x)$ and $(s(n))_{n\in\NN}\in(\KK^{\aleph_0}, \sigma')$ as previously defined. We say that $(s(n))_{n\in\NN}$ is a zero of $p$ if, under the unique homomorphism of difference rings $\mathcal{D}_{\sigma}(\KK,x)\longrightarrow~(\KK^{\aleph_0},\sigma')$ given by the extension that sends $x$ to $(s(n))_{n\in\NN}$, $p$ is sent to $0$. We interpret the fact that $(s(n))_{n\in\NN}$ is a zero of $p$ by ``$p(s(n))=0$ for all non-negative integers $n$'' or ``$p((s(n))_{n\in\NN})=0$'' (or simply $p(s(n))=0$ with unspecified $n$).
\end{definition}

Several sequences are zeros of difference polynomials, but in some cases, being a zero of a difference polynomial does not particularly add value to the knowledge of the sequence. According to \Cref{def:defzero}, even the sequence of Bernoulli numbers satisfies an algebraic difference equation. Indeed, consider the difference polynomial
\begin{equation}\label{eq:berndiffpol}
    p\coloneqq 5 \,\sigma^3(x)\,x-6 \,\sigma^2(x)\sigma(x)+\sigma(x)\,x.
\end{equation}

Recall that the $n$th Bernoulli number $B_n$ may be computed by the formulas:

\begin{equation}
    \begin{split}
        &B_0=1,\, B_{2n} = \frac{(-1)^{n+1} 2\,(2n)!}{(2\pi)^{2n}}\zeta(2n),\, n\geq 1,\label{eq:zeta}\\
        &B_1=-\frac{1}{2},\, B_{2n+1} = 0,\, n\geq 1,
    \end{split}
\end{equation}

where $\zeta$ is the Riemann Zeta function. One verifies that for all $n\in\NN$, $p(B_n)=0$. However, the only terms of $(B_n)_{n\in\NN}$ that can be deduced from $p$ are the terms $B_{2n}, n\leq 2$ and $B_{2n+1},n\geq 0$. When such a phenomenon happens for a given sequence and a difference polynomial, we say that the sequence is not a {\it generic zero} of that difference polynomial. Let us make this statement more precise.

\begin{definition} Let $p\in\mathcal{D}_{\sigma}(\KK,x)$. The initial $\mathcal{I}_p$ of $p$ is the leading coefficient of $p$ viewed as a univariate polynomial in $\sigma^{\ord(p)}(x)$.
\end{definition}

\begin{proposition}\label{prop:orddegI} Let $p\in\mathcal{D}_{\sigma}(\KK,x)\setminus \KK$. We have $\ord(\mathcal{I}_p)<\ord(p)$ and $\deg(\mathcal{I}_p)<\deg(p)$.
\end{proposition}
\begin{proof} The inequality for the order is immediate from the definition. For the degree, let $r=\ord(p)$ and $m\coloneqq\deg_{\sigma^r(x)}(p)>0$ be the degree of $p$ in $\sigma^r(x)$. From the definition of $\mathcal{I}_p$, it follows that there exists $p_1 \in\mathcal{D}_{\sigma}(\KK,x), \deg_{\sigma^r(x)}(p_1)<m, \ord(p_1)\leq r$ such that 
\begin{equation}\label{eq:indegp1}
p=\mathcal{I}_p\left(\sigma^{\ord(p)}(x)\right)^m+p_1,
\end{equation}
showing that $\deg(p)\geq \deg(\mathcal{I}_p)+m>\deg(\Initp)$.
\end{proof}
\begin{definition} Let $p\in\mathcal{D}_{\sigma}(\KK,x)$. A sequence $(s(n))_{n\in\NN}$ is a generic zero of $p$ if it is a zero of all $P$ in the difference ideal $<p>$ and not a zero of any $P\notin <p>$.
\end{definition}
\begin{proposition}\label{prop:genzeroInit} If a sequence $(s(n))_{n\in\NN}$ is a generic zero of an algebraically irreducible difference polynomial $p\in\mathcal{D}_{\sigma}(\KK,x)$, then $\mathcal{I}_p((s(n))_{n\in\NN})\neq 0$.
\end{proposition}
\begin{proof} If $\mathcal{I}_p((s(n))_{n\in\NN}) = 0$ then $\mathcal{I}_p\in <p>$. Since $p$ is irreducible, there exist difference polynomials $q_0,q_1,\ldots,q_k\in\mathcal{D}_{\sigma}(\KK,x)$ such that 
\begin{equation}\label{eq:lincombInit}
   \Initp = \sum_{j=0}^k q_j\sigma^j(p). 
\end{equation}
Using an orderly ranking on $R\coloneqq \KK[x,\sigma(x),\ldots,\sigma^{k+\ord(p)}(x)]$, the coefficients $q_j$'s can be interpreted as the quotients in a reduction (division) of $\Initp$ with respect to the \textit{triangular set} $\{\sigma^j(p),j=0\ldots,k\}$ generating the truncation of $<p>$ in $R$. By the reduction algorithm, we must have that
\[\deg(\Initp) = \deg(q_k\,\sigma^k(p))=\deg(q_k)+\deg(p),\]
using the correspondence condition between the leading terms on both sides in \eqref{eq:lincombInit}.

However, by \Cref{prop:orddegI} we know that $\deg(\Initp)<\deg(p)$. Therefore we must have $\deg(q_k)<~0$, thus $q_k=0$. Repeating this reasoning with the remaining $q_j$'s, $0\leq j < k$, leads to $\Initp = 0$, a contradiction since $p$ is irreducible. 
\end{proof}

The following theorem may be regarded as a corollary of R. Cohn's existence theorem for ordinary difference polynomials.
\begin{theorem}{\cite[Theorem IV]{cohn1948manifolds}}\label{th:cohn} Every irreducible difference polynomial in $\mathcal{D}_{\sigma}(\KK,x)\setminus \KK$ has a generic zero.
\end{theorem}
 Following Cohn's original proof, the core idea of a generic zero—as illustrated by \Cref{prop:genzeroInit} and \Cref{th:cohn}—is that all subsequent terms of such a sequence can be computed from a corresponding difference polynomial $p \in \Dsigma$ and its first $\ord(p)$ initial terms. This property, however, is the essential problem that hampers the definition of the Bernoulli sequence using the irreducible polynomial $p$ from \eqref{eq:berndiffpol}.

The classical definition of a generic zero does not fully capture this idea of computability. We resolve this by introducing the class of D-algebraic sequences. To simplify the understanding of our main definition, we first consider the case of irreducible difference polynomials.

 \begin{definition}[D-algebraicity with irreducible difference polynomial]\label{def:dalgdefirr} Let $p\in\Dsigma$ be irreducible. A sequence $(s(n))_{n\in\NN}$ is difference-algebraic (or D-algebraic) over $\KK$ with defining polynomial $p$, if $\Initp(s(n))\neq 0$ for sufficiently large $n$.
 \end{definition}

Whether the Bernoulli sequence is D-algebraic or not is not proven in this paper. However, from \Cref{def:dalgdefirr} we know that the difference polynomial $p$ in \eqref{eq:berndiffpol}
cannot be used to define the Bernoulli sequence since $\Initp(B_{2n+1})=0, n\geq 1$. In other words, for an irreducible difference polynomial $p$ to define a D-algebraic sequence $(s(n))_{n\in\NN}$, we must have $|\{n\in\NN, \Initp(s(n))=0\}|<\infty.$ We refer to this condition as {\it the regularity condition}.

In general, to show that a given sequence is D-algebraic, it is enough to show that all its terms can be obtained by finding the roots of some univariate polynomials resulting from evaluating a difference polynomial with its previous terms. Generalizing \Cref{def:dalgdefirr} for arbitrary difference polynomials, we obtain the definition below.

\begin{definition}\label{def:dalgdef} Let $p\in\Dsigma$, and $p=p_0p_1\cdots p_{\ell},$ its decomposition into irreducible components. Let $\NN_{p_j}(s(n))\coloneqq \{n\in\NN,\, p_j(s(n))=0\}.$ A sequence $(s(n))_{n\in\NN}$ is difference-algebraic (or D-algebraic) over $\KK$ with defining difference polynomial $p$, if:
\begin{enumerate}
    \item For all $j\in\{0,\ldots,\ell\}$, $|\NN_{p_j}(s(n))|=\infty;$
    \item For all $j\in\{0,\ldots,\ell\}$, $|\{n\in\NN_{p_j}(s(n)),\, \Init_{p_j}(s(n))=0\}|<\infty;$
    \item For all $n\in\NN$, there exists a unique $j\in\{0,\ldots,\ell\}$ such that $p_j(s(n))=0$.
\end{enumerate}
\end{definition}

Observe that when $\ell=0,$ \Cref{def:dalgdef} recovers \Cref{def:dalgdefirr}: conditions $1$ and $3$ together ensure that $(s(n))_{n\in\NN}$ is a generic zero of $p_0$, and condition $2$ corresponds precisely to the regularity condition. Given the ``switching'' behavior of a D-algebraic sequence's terms among the irreducible factors of its defining difference polynomial (as described in \Cref{def:dalgdef}), we may refer to such sequences as \textit{polymorphic} D-algebraic sequences. Nevertheless, we will not explicitly distinguish between the two in this paper. From the next section onwards, our focus will primarily be on D-algebraic sequences defined by irreducible difference polynomials.

\begin{example}[``A D-algebraic continuation of Bernoulli numbers'']\label{ex:bernoullicont} The sequence $(b(n))_{n\in\NN}$ defined as
\begin{equation}\label{eq:sumconsec}
\begin{split}
    &b\! \left(n+3\right)=-\frac{6 b\! \left(n\right) b\! \left(n+1\right)+5 b\! \left(n\right) b\! \left(n+2\right)-6 b\! \left(n+1\right)^{2}-30 b\! \left(n+1\right) b\! \left(n+2\right)}{25 b\! \left(n\right)},\, n\geq 0,\\
    &b\left(0\right)\coloneqq \frac{1}{2};\, b\left(1\right)\coloneqq -\frac{1}{3}\,; b\left(2\right)\coloneqq \frac{1}{6}, 
\end{split}
\end{equation}
is a D-algebraic sequence whose terms are defined by the sum of two consecutive terms of a D-algebraic zero of \eqref{eq:berndiffpol}. Its first three initial terms are $B_0+B_1,\, B_1+B_2,$ and $B_2+B_3$, obtained with the Bernoulli sequence. This shows that the sequence 
\[(u(n))_{n\in\NN}=\left(1,-\frac{1}{2},\frac{1}{6},0,-\frac{1}{30},0,\frac{91}{3750},0,-\frac{423241}{11718750},0,\frac{85414689451}{915527343750},0,\ldots\right),\] 
defined by 
\begin{equation}
    u(n)\coloneqq \begin{cases}
    b(n)\quad \text{if}\quad n=2k,\, k\geq1,\\
    B_n\quad \text{ otherwise, } 
    \end{cases}
\end{equation}
is a generic zero of \eqref{eq:berndiffpol}, which, as the Bernoulli sequence, cannot be defined by it. By adjusting \eqref{eq:sumconsec}, we deduce the following difference polynomial to define $(u(n))_{n\in\NN}$ as a D-algebraic sequence.
\begin{equation}
    \sigma^2(x)\,\left(25\,x \sigma^4(x) + 11\,x\sigma^2(x) - 36\,(\sigma^2(x))^2+\sigma^3(x)\right).
\end{equation}
For $n\geq 0$, the first factor from the left helps in computing zero terms at odd indices, and the second factor helps to compute the nonzero terms.
\QEDA
\end{example}

\begin{example}\label{ex:interlaceeg1} Consider $p\coloneqq\left(2 x\sigma(x)+x+\sigma(x)\right) \left(2 \sigma(x) x-x-\sigma(x)\right)$. Using $p$ and the initial term $s(0)\coloneqq 1$, we define the recursion:
\begin{equation}\label{eq:sinterlace}
    s(n+1)\coloneqq \begin{cases}
                     \frac{s\left(n\right)}{2\cdot s\left(n\right)-1},\, \text{if}\,\,\, n\,\,\, \text{is odd,}\\
                     -\frac{s\left(n\right)}{2\cdot s\left(n\right)+1},\, \text{if}\,\,\, n\,\,\, \text{is even.}
                     \end{cases}
\end{equation}
The sequence $(s(n))_{n\in\NN}$ is D-algebraic and we have $s(n)=\frac{(-1)^n}{2n+1},$ for all $n\in\NN$. This can be shown by checking the initial term and substituting $\frac{(-1)^n}{2n+1}$ for $n=2k$ and $n=2k+1$ into \eqref{eq:sinterlace} to verify that $s(n+1)=\frac{(-1)^{n+1}}{2(n+1)+1}.$ Using the algorithm from \cite{teguia2024rational}, one shows that $(s(n))_{n\in\NN}$ can also be defined with the difference polynomial $q\coloneqq (2\,x+\sigma(x))\,\sigma^2(x)+x\sigma(x).$ A question of interest for future studies arose in the understanding of the relation between $p$ and $q$ for $(s(n))_{n\in\NN}$.   \QEDA
\end{example}

We view $\KK^{\aleph_0}$ as the set of \textit{affine D-algebraic sequences}, which we will simply refer to as D-algebraic sequences, as these are our sole focus. Although arithmetic with coordinates does not have a particular interest in algebraic geometry, we anticipate a comparative study relating {\it projective D-algebraic sequences}, which we may write $\mathbb{P}^{\aleph_0}(\KK)$, and $\KK^{\aleph_0}$, especially on their defining difference polynomials. 

\begin{remark}
     With the above definition of a difference-algebraic sequence, it is convenient to look at an irreducible difference polynomial $p\in\mathcal{D}_{\sigma}(\KK,x)$ as the {\it specialized} (see \cite{pogudin2020solving}) difference polynomial $p\in \DD_s\coloneqq \KK[\sigma^j(s(n)) | j\in\NN]$, where $s(n)$ (symbolically) represents the general term of a D-algebraic zero of $p$. This will be used to avoid lengthy notations. We use the automorphism $\sigma$ with the assumption that it is extended accordingly for $\KK^{\aleph_0}$.
\end{remark}

\section{Arithmetic}

We adapt the construction from \cite[Section 2.2]{teguia2025arithmetic} to the difference case. We focus on irreducible difference polynomials. 

Let $N\in\NN\setminus \{0\}$. We consider $N$ D-algebraic sequences $(s_i(n))_{n\in\NN},$ each defined by a given irreducible difference polynomial $p_i\in\Dsigma,$ of order $n_i\in\NN.$ Let $g,h\in\KK[X_1,\ldots,X_N]$ be two coprime polynomials, and define $f\coloneqq \frac{g}{h}$. We aim to construct a difference polynomial $q$ such that the sequence defined by  $t(n)=f(s_1(n),\ldots,s_N(n))$ is a D-algebraic zero of $q.$ 

For $t(n)$ to be well-defined for sufficiently large $n$, the denominator $h(s_1(n),\ldots,s_N(n))\neq 0$ must be nonzero. Combining this requirement with the regularity condition for each individual sequence $(s_i(n))_{n\in\NN}$, we deduce the following critical prerequisite for the construction:
\begin{equation}\label{eq:arithmcond}
    Q(s_1(n),\ldots,s_N(n))\coloneqq h(s_1(n),\ldots, s_N(n))\, \prod_{i=1}^N\Init_{p_i}(s_i(n)) \,\neq\, 0,\,\, \text{for sufficiently large}\,\, n.
\end{equation}
For our algorithmic approach, we utilize the symbolic representation $p_i(s_i(n),\ldots,s_i(n+n_i))\in\DD_{s_i}.$ This signifies that we are treating the $s_i(n)$'s as symbolic variables, rather than focusing on the specific numerical values dependent on $n$. We use these $p_i$'s and $f$, combined with condition \eqref{eq:arithmcond}, to construct a difference polynomial $q(t(n),\ldots,t(n+n_q))\in\DD_t, n_q\in\NN.$ This $q$ is designed to vanish when evaluated at $t(n)=f(s_1(n),\ldots,s_N(n))$, provided that relevant shifts of $Q(s_1(n),\ldots,s_N(n))$ are nonzero. In practice, due to the D-algebraic nature of the sequences $s_i$'s, this condition simplifies to assuming that corresponding shifts of $h(s_1(n),\ldots,s_N(n))$ are nonzero.

While precisely determining the indices $n$ for which $Q(s_1(n),\ldots,s_N(n))$ might vanish is generally challenging, this underscores the importance and utility of an algorithmic method that can derive $q$ under the assumption of condition \eqref{eq:arithmcond}'s general validity. We also mention that the initial terms of the sequences play a crucial role in the practical application of our method.

\subsection{Algorithm}\label{sec:algo}

We use the same notations from the introductory paragraph of this section. We define the indeterminates
\begin{equation}\label{eq:indets}
\begin{split}
    &w_{j+1}(n)=s_1(n+j),\,\, \text{ for }\,\, 0\leq j < n_1\\
    &w_{j+1}(n)=s_2(n+j),\,\, \text{ for }\,\, n_1\leq j < n_1+n_2\\
    &\vdots\\
    &w_{j+1}(n)=s_N(n+j),\,\, \text{ for }\,\, M-n_N\leq j < M.
\end{split}
\end{equation}
Observe that for $j\in{1,\ldots,M} \setminus \{n_1,n_1+n_2,\ldots,M\}$, $\sigma(w_j(n))=w_{j+1}(n)$, and $\sigma(w_{M_{i+1}}(n))$ satisfies
\begin{equation}\label{eq:eqindets2}
    p_{i+1}\left(w_{M_i+1}(n),\ldots,w_{M_{i+1}}(n),\sigma(w_{M_{i+1}}(n))\right)=0,
\end{equation}
where $M_i\coloneqq \sum_{j=0}^i n_i$, for $i=0,\ldots,N-1$ and $n_0=0.$ The index variable can now be regarded implicitly, i.e., we can write $w_j$ for $w_j(n)$ since its shifts are well understood.

For $i=0\ldots,N-1$, we write $c_{m_i+1}=\Init_{p_{i+1}}$, and $p_{i+1}$ as (see \eqref{eq:indegp1})
\begin{equation}\label{eq:eqpjmj}
    c_{m_{i+1}}\,{\sigma(w_{M_{i+1}})}^{m_j} + p_{i+1,1}\left(w_{M_i+1},\ldots,w_{M_{i+1}},\sigma(w_{M_{i+1}})\right).
\end{equation}
 We can now view the problem as resulting from the following {\it radical-rational dynamical} system
\begin{equation}\label{eq:radratdyn}
   \begin{cases}
    \sigma(\mathbf{w})^{\mu} = \mathbf{A}(\mathbf{w}) + \mathbf{E}_{\mathbf{A}}(\sigma(\mathbf{w}))\\
    z = B(\mathbf{w})
    \end{cases}\coloneqq \begin{cases}
    &{\sigma(w_1)}^{\mu_1}=A_1(w_1,\ldots,w_M) + E_{A_1}(\sigma(w_1))\\
    &\vdots\\
    &{\sigma(w_{M})}^{\mu_M}=A_M(w_1,\ldots,w_M)+E_{A_M}(\sigma(w_M))\\
    &z = B(w_1,\ldots,w_M)
\end{cases}, \,\quad\,
(\mathcal{M}_f) 
\end{equation}
where 
\begin{itemize}
    \item $\mu_i,i=1,\ldots,M$, is either $1$ or one of the $m_i, i=1,\ldots,N$; 
    \item $A_i$ (with numerator $a_i$) is either $w_{i+1}$ or the part of $p_{i,1}/c_{m_i}$ that is free of $\sigma(w_i)$ $i=1,\ldots,N$;
    \item $E_{A_i}(\sigma(w_i))$ (with numerator $e_{a_i}$) is either $0$ or the part of $p_{i,1}/c_{m_i}$ that contains $\sigma(w_i)$ (in its numerator), $i=1,\ldots,N$;
    \item $B(w_1,\ldots,w_M)=f(w_1,w_{n_1+1},\ldots,w_{\sum_{i=1}^{N-1}n_i+1})$.
\end{itemize}

We refer to $M$ as the dimension of $(\mathcal{M}_f)$.

\begin{remark} When the $N$ given difference polynomials are rationalizing, i.e., $m_i=1,$ $i=1,\ldots, N$, and $E_{A_i}=0$, $(\mathcal{M}_f)$ is a classical rational dynamical system often considered in theoretical computer science \cite{clemente2023rational}.
\end{remark}

Let $Q$ be the common denominator of all denominators in the system. Without loss of generality, we assume that all $A_i$ and $E_{A_i}$ are written with $Q$ as the denominator and consider that $B=b/Q$. This $Q$ carries the condition from \eqref{eq:arithmcond}. We may regard $(\mathcal{M}_f)$ as a system of difference polynomials on the multivariate ring of difference polynomials $\DD_{\mathbf{w},z}\coloneqq\DD_{w_1,\ldots,w_M,z}$. The corresponding difference ideal is

\begin{equation}\label{eq:eq9}
    I_{\mathcal{M}_f}\coloneqq \langle Q{\sigma(\mathbf{w})}^{\mu}-\mathbf{a}(\mathbf{w}) - \mathbf{e}_{\mathbf{a}}(\sigma(\mathbf{w})),\, Q\,z-b(\mathbf{w}) \rangle \colon H^{\infty} \subset \DD_{\mathbf{w},z},
\end{equation}
where $H\coloneqq \{Q,\sigma(Q),\sigma^2(Q),\ldots\}$, and ``$\colon H^{\infty}$'' denotes the saturation with $H$.  

\Cref{algo:Algo1} gives the steps of our approach for the arithmetic of D-algebraic sequences.

\begin{algorithm}[ht]\caption{Arithmetic of difference-algebraic sequences}\label{algo:Algo1}
    \begin{algorithmic} 
    \\ \Require $N$ difference polynomials $p_i\in\DD_{s_i}$ of order $n_i$ and a function $f\in\KK(X_1,\ldots,X_n)$.
    \Ensure A difference polynomial of order at most $M\coloneqq n_1+\cdots+n_N$ that vanishes at $f(s_1,\ldots,s_N)$ for appropriate values of $n$, where each $s_i$ is a D-algebraic zero of $p_i$.
    \begin{enumerate}
    \item Construct $(\mathcal{M}_f)$ from the input $p_1,\ldots,p_N$ as in \eqref{eq:radratdyn}.
    \item Denote by $\mathcal{E}$ the set  
        \begin{eqnarray*}
             \mathcal{E}&\coloneqq& \lbrace Q\,{\sigma(\mathbf{w})}^{\mu}-\mathbf{a}(\mathbf{w})-\mathbf{e}_{\mathbf{a}}(\mathbf{w}'), Q \, z - b(\mathbf{w})\rbrace \\
                        &=& \big\{ Q\,{\sigma(w_i)}^{\mu_i}-a_i(w_1,\ldots,w_M) - e_{a_i}(\sigma(w_i)), i=1,\ldots,M,\, Q\,z - b(w_1,\ldots,w_M)\big\}, 
        \end{eqnarray*}


        \item Compute the first $M-1$ shifts (application of $\sigma$) of all polynomials in $\mathcal{E}$ and add them to $\mathcal{E}$.
        \item Compute the $M$th shift of $Q\, z - b(w_1,\ldots,w_M)$ and add it to $\mathcal{E}$. We are now in the ring $\KK[\sigma^{\leq M}(w_1),\ldots,\sigma^{\leq M}(w_M),\sigma^{\leq M}(z)]$, which contains all differential polynomials of order at most $M$ in $\DD_{\mathbf{w},z}$.
        \item Let  $I\coloneqq\langle \mathcal{E} \rangle \subset \KK[\sigma^{\leq M}(w_1),\ldots,\sigma^{\leq M}(w_M),\sigma^{\leq M}(z)]$ be the ideal generated by the elements of~$\mathcal{E}$.
        \item Let $H\coloneqq \{Q,\sigma(Q),\ldots,\sigma^M(Q)\}$.
        \item Update $I$ by its saturation with $H$, i.e, $I\coloneqq I\colon H^{\infty}$.
        \item\label{step:elimalg1} Compute the elimination ideal $I \cap \KK[\sigma^{\leq M}(z)]$. From the resulting Gr\"{o}bner basis, choose a polynomial $q$ of the lowest degree among those of the lowest order. 
        \item Return $q$.
    \end{enumerate}	
    \end{algorithmic}
\end{algorithm}

To prove the correctness of \Cref{algo:Algo1}, we must show that $\langle \sigma^{\leq M} \left(I_{\mathcal{M}_f}\right) \rangle \cap \KK[\sigma^{\infty}(z)]$ in step \ref{step:elimalg1} is a non-trivial elimination ideal. The following theorem establishes this fact.

\begin{theorem}\label{th:theo1} On the commutative ring $\KK[\sigma^{\infty}(\mathbf{w}),\sigma^{\infty}(z)]$, seen as a polynomial ring in infinitely many variables, consider the lexicographic monomial ordering corresponding to any ordering on the variables such that
\begin{itemize}
\item[(i.)] $\sigma^{j_1}(z)\succ \sigma^{j_2}(w_i)$, $i, j_1, j_2\in \NN$,
\item[(ii.)] $\sigma^{j+1}(z)\succ \sigma^j(z)$, $\sigma^{j+1}(w_{i_1})\succ \sigma^j(w_{i_2})$, $i_1, i_2,j\in\NN$.
\end{itemize}
Then the set $\mathcal{E}\coloneqq\left\lbrace Q{\sigma(\mathbf{w})}^{\mu}-\mathbf{a}(\mathbf{w})-\mathbf{e}_{\mathbf{a}}(\sigma(\mathbf{w})),\, Qz-b(\mathbf{y}) \right\rbrace$ is a triangular set with respect to this ordering. Moreover,
\begin{equation}\label{eq:eq10}
    \langle \sigma^{\leq M} \left(I_{\mathcal{M}_f}\right) \rangle \cap \KK[\sigma^{\infty}(z)] \neq \langle 0 \rangle.
\end{equation}
\end{theorem}
\begin{proof} First, let us mention that the system $(\mathcal{M}_f)$ is consistent by \Cref{th:theo1} from the individual input difference polynomials. The leading monomials of 
\begin{align*}
    &\sigma^j\left(Q\,\sigma(w_i)^{\mu_i}-a_i(\mathbf{w})-e_{a_i}(\sigma(w_i))\right) \,\text{ and }\, \sigma^j\left(Q\,z-b(\mathbf{w})\right)
\end{align*}
 in the ring $\KK[\sigma^{\infty}(\mathbf{w}),\sigma^{\infty}(z)]$ have highest variables $\sigma^{j+1}(w_i)$ and $\sigma^j(z)$, respectively. Since these variables are all distinct, by definition (see \cite[Definition 4.1]{hubert2003notes}), we deduce that $\mathcal{E}$ is a consistent triangular set with coefficients in the field $\KK$. As a triangular set, $\mathcal{E}$ defines the ideal $\langle \mathcal{E}\rangle:H^{\infty}=I_{\mathcal{M}_f}$, where $H~\coloneqq~\{Q,\sigma(Q),\ldots,\sigma^M(Q)\}$. Therefore by \cite[Theorem 4.4]{hubert2003notes}), all associated primes of $\sigma^{\leq M}(I_{\mathcal{M}_f})$ share the same transcendence basis given by the non-leading variables $\{w_1,\ldots,w_M\}$ in $\sigma^{\leq M}(I_{\mathcal{M}_f})$. Thus the transcendence degree of $\KK[\sigma^{\leq M}(\mathbf{w}),\sigma^{\leq M}(z)]/\langle \sigma^{\leq M}\left(I_{\mathcal{M}_f}\right) \rangle$ over $\KK$ is $M$. However, the transcendence degree of $\KK(x)[\sigma^{\leq M}(z)]$ is $M+1$. Hence we must have $\langle \sigma^{\leq M}\left(I_{\mathcal{M}_f}\right) \rangle \cap \KK[\sigma^{\infty}(z)] \neq \langle 0 \rangle$.
\end{proof}

Observe that $M$ is the minimal integer for which the argument of the proof of \Cref{th:theo1} holds. 

\subsection{Some closure properties and examples}

We present some immediate consequences of the result from the previous subsection, including the ring structure of $\KK^{\aleph_0}$, induced by addition and multiplication. We show how to use \Cref{algo:Algo1} beyond arithmetic properties. The central fact resulting from \Cref{th:theo1} is that building a dynamical system in the form of $\mathcal{M}_f$ in \eqref{eq:radratdyn} is enough to show the existence of a difference polynomial vanishing at sequences' transformations encoded by some $f$. We exploit the fact that the form of the system is unchanged by adding new variables and raising them to some powers on its left-hand side.

We say that a sequence $(s(n))_{n\in\NN}$ (or simply $(s(n))_n$ if there is no ambiguity) is D-algebraic of order $r$ if it is defined by an irreducible difference polynomial of order $r$.

\begin{corollary}\label{cor:cor1} Let $(u(n))_n$ and $(v(n))_n$ be two difference-algebraic sequences of order $r_1$ and $r_2$, respectively, and $i,j\in\NN, i< r_1,\, j< r_2$. Then, the following sequences are difference-algebraic of order $r$:
\begin{enumerate}
    \item addition: $(u(n+i)+v(n+j))_n$, with  $r\leq r_1+r_2$;
    \item multiplication: $(u(n+i)\,v(n+j))_n$, with $r\leq r_1+r_2$;
    \item division: $(1/v(n+j))_{n\geq n_0}$, with $r\leq r_2$, $v(n)\neq 0$ for all $n\geq n_0\in\NN$;
    \item taking radicals: $(\sqrt[N]{u(n+i)}\,)_n,N\in\NN$, with $r\leq r_1$;
    \item partial product: $(\prod_{k=k_0}^n u(k+i))_n, k_0\in\NN$, with  $r\leq r_1+1;$
    \item partial sum: $(\sum_{k=k_0}^n u(k+i))_n,\, k_0\in\NN$, with $r\leq r_1+1$;
\end{enumerate}
\end{corollary}
\begin{proof} The proofs of these properties are deduced from constructions of radical-rational dynamical systems $(\mathcal{M}_f)$ as in \eqref{eq:radratdyn} for some rational functions, here denoted $f$. We provide details for $i=j=0$ as the case $i,j\neq 0$ only implies a different encoding of $f$ with the variables of the dynamical system.
\begin{enumerate}
    \item $f(X,Y)\coloneqq X+Y$;
    \item $f(X,Y)\coloneqq X\,Y$;
    \item $f(X)\coloneqq 1/X$ and $(\mathcal{M}_f)$ built with the defining difference polynomial of $(v(n))_n$ only;
    \item This follows from the fact that the correctness of \Cref{algo:Algo1} is unchanged with having $z$ replaced by $z^N$, $N\in\NN$. Thus, $(\mathcal{M}_f)$ is built with $f(X)=X$ such that the output equation writes $z^N=w_1$. Note, however, that this is an elementary fact. Indeed, one verifies that the difference polynomial obtained by substituting $(\sigma^j(x))^i$ by $(\sigma^j(x))^{i\,N}$ in $p$ has $(\sqrt[N]{u(n)}\,)_n$ as a zero.
    \item\label{proof:series} Let us denote by $(t(n))_n$, the partial product of $(u(n))_n$. It is defined by the recursion $t(n+1)=t(n)\, u(n)$, with $t(0)=1$. First, consider the radical-rational dynamical system $(\mathcal{M}_f)$ constructed from the difference polynomial of $(u(n))_n$, with $f$ unspecified. At this stage, we know that for any $f$, we have $r_1$ components in the system. To take $(t(n))_n$ into account, we add a new variable $w_{r_1+1}$, such that $\sigma(w_{r_1+1})=w_{r_1+1}\,w_1$, which represents the recursion. Then, we choose the function $b$ of the system as $b=w_{r_1+1}$, which tells \Cref{algo:Algo1} that we want a difference polynomial for $(t(n))_n$. Hence, by construction, $(t(n))_n$ is D-algebraic of order at most $r_1+1$ as claimed. Some constraints of the resulting difference polynomial may define the value of $k_0$. 
    \item For the partial sum, say $(s(n))_n$, one considers the recursion $s(n+1)=s(n)+u(n)$, with $s(0)=0$, and proceeds in the same way as with the partial product. One can verify that the difference polynomial obtained by substituting $x$ by $\sigma(x)-x$ in the defining difference polynomial of $(u(n))_n$ belongs to the difference ideal of the output of the algorithm.
\end{enumerate}
\end{proof}
Applications of \Cref{algo:Algo1} reveal that a dynamical system in the form of $(\mathcal{M}_f)$ from \eqref{eq:radratdyn} serves as a certificate for computing an algebraic difference equation (or difference polynomial) associated with a given problem. 

In the examples below, the assumption related to \eqref{eq:arithmcond} is easily verified. Therefore, we will not pay particular attention to that required condition.
\begin{example}\label{ex:closure} Let $p,q\in\mathcal{D}_{\sigma}(\KK,x)$, such that $p=\sigma^2(x)-\sigma(x)\,x$, $q=\sigma(x)-x^2-x$, specialized with the generic solution $u(n)$ and $v(n)$, respectively. With the initial values $u(0)=1$, $u(1)=k\in\KK$, the algebraic difference equation associated to $p$ yields the sequence of general term $k^{F_n}$, where $F_n$ is the $n$th Fibonacci number (see, for instance, \href{https://oeis.org/A000301}{A000301} from \cite{sloane2003line}). With $v(0)=1,v(1)=2$, $v(n)$ is known to denote the number of ordered trees having nodes of outdegree $0,1,2$ and such that all leaves are at level $n$ (see \href{https://oeis.org/A007018}{A007018}). Let us find algebraic difference equations satisfied by $(u(n)/v(n))_n$, $(u(n)\,v(n))_n$, and $(\sum_{k=0}^n u(k))_n$. In all these cases, we will write the resulting equation with the undetermined term $s(n)$. 
\begin{enumerate}
    \item $(u(n)/v(n))_n$.

    We use the same notations from the previous subsection. The corresponding dynamical system is given by
    \begin{equation}\label{eq:sys1}
        \begin{cases}
            \sigma(w_1) = w_2\\
            \sigma(w_2) = w_1\,w_2\\
            \sigma(w_3) = w_3^2 + w_3\\
            z = \frac{w_1}{w_3}
        \end{cases}.
    \end{equation}
    After elimination, we obtain the following principal ideal
    \begin{equation}
    \begin{split}
        &\Big\langle\sigma(z)^{4} \sigma^3(z)^{2} z^{4}-z^{3}\sigma(z)^{4}\sigma^2(z)^{2}\sigma^3(z)-z^{3}\sigma(z)^{3}\sigma^2(z)\sigma^3(z)^{2}-z^{2} \sigma(z)^{3} \sigma^2(z)^{3} \sigma^3(z)\\
        &+z \sigma(z)^{3} \sigma^2(z)^{5}-z^{2} \sigma(z)^{2} \sigma^2(z)^{2} \sigma^3(z)^{2}+2 z \sigma(z)^{2} \sigma^2(z)^{4} \sigma^3(z)+\sigma(z)^{2} \sigma^2(z)^{6}\\
        &+z \sigma(z) \sigma^2(z)^{3} \sigma^3(z)^{2}+2 \sigma(z) \sigma^2(z)^{5} \sigma^3(z)+\sigma^2(z)^{4} \sigma^3(z)^{2}\Big\rangle,
    \end{split}
    \end{equation}
    where $\sigma^j(z)^i$ is understood as $(\sigma^j(z))^i$. Hence the algebraic difference equation
    \begin{equation}\label{eq:div1}
        \begin{split}
        &s(n+1)^{4}s(n+3)^{2}s(n)^{4}-s(n+1)^{4} s(n+3) s(n)^{3} s(n+2)^{2}+s(n+3)^{2} s(n+2)^{4}\\
        &-s(n+1)^{3} s(n+3)^{2} s(n)^{3} s(n+2)+s(n+1)^{3} s(n) s(n+2)^{5} +s(n+1)^{2} s(n+2)^{6}\\
        &+2 s(n+1)^{2} s(n+3) s(n) s(n+2)^{4}-s(n+1)^{2} s(n+3)^{2} s(n)^{2} s(n+2)^{2}\\
        &+s(n+1) s(n+3)^{2} s(n) s(n+2)^{3}+2 s(n+1) s(n+3) s(n+2)^{5}\\
        &-s(n+1)^{3} s(n+3) s(n)^{2} s(n+2)^{3}=0,\\
        \end{split}
    \end{equation}
    of order $3$ and degree $10$.
    \item $(u(n)\,v(n))_n$.
    
    The system is similar to \eqref{eq:sys1} with the last equation replaced by $z = w_1\,w_3$. We obtain a principal ideal with the following associated equation:
    \begin{equation}\label{eq:mul1}
    \begin{split}
        &s(n+2)^{2} s(n+1)^{4} s(n)^{4}+2 s(n+3) s(n+2) s(n+1)^{3} s(n)^{4}+s(n+2)^{3} s(n+1)^{3} s(n)^{3}\\
        &+s(n+3)^{2} s(n+1)^{2} s(n)^{4}+2 s(n+3) s(n+2)^{2} s(n+1)^{2} s(n)^{3}-s(n+2)^{4} s(n+1)^{2} s(n)^{2}\\
        &+s(n+3)^{2} s(n+2) s(n+1) s(n)^{3}-s(n+3) s(n+2)^{3} s(n+1) s(n)^{2}-s(n+2)^{5} s(n+1) s(n)\\
        &-s(n+3) s(n+2)^{4} s(n)+s(n+2)^{6}=0.
    \end{split}
    \end{equation}

    Observe that \eqref{eq:mul1} is also satisfied by the sequence $(v(n)/u(n))_n$ since $(1/u(n))_n$ is also a zero of $p$.

    \item $(\sum_{k=0}^n u(k))_n$.

    The corresponding system writes

    \begin{equation}\label{eq:sys2}
        \begin{cases}
            \sigma(w_1) = w_2\\
            \sigma(w_2) = w_1\,w_2\\
            \sigma(w_3) = w_3 + w_1\\
            z = w_3
        \end{cases}.
    \end{equation}
    Here again, we obtain a principal ideal. The associated equation is given by
    \begin{dmath}
        s(n) s(n+1)-s(n) s(n+2)-s(n+1)^{2}+s(n+1) s(n+2)+s(n+2)-s(n+3)=0.
    \end{dmath}
\end{enumerate}
The fact that we obtain principal ideals in all three examples is a common situation encountered in the differential case. It is explained in \cite[Remark 4]{dong2023differential} that such an ideal is generally ``almost principal''. \QEDA
\end{example}

We now deduce a result for sequence acceleration that follows from \Cref{cor:cor1}.

\begin{corollary} Let $(s(n))_{n\in\NN}$ be a D-algebraic sequence of order $r>2$. Assume that $\Delta^2s(n)\neq 0$ for sufficiently large $n$. The Aitken transformation of $(s(n))_{n\in\NN}$, say $(t(n))_{n\in\NN}$, of general term $t(n)\coloneqq s(n)-\frac{\left(\Delta s(n)\right)^2}{\Delta^2s(n)}$ is D-algebraic of order $r$.
\end{corollary}

\begin{example}[An example of Aitken acceleration \cite{griffiths2016100}]\label{ex:aitarithmexpl} In his {\it example of Aitken acceleration}, Greffins considered a logical extension of Fibonacci's rabbit problem in which two neighboring rabbit populations are, in addition to growing, competing for units of grassland. The probabilistic process that results from this scenario potentially gives rise
to two random walks involving the Fibonacci numbers. We focus on the second, whose expectation is given by
\begin{equation*}
\EE_{n} = -2 \sum_{k=2}^n \frac{1}{F_{k}F_{k+1}},\, n\geq 2,
\end{equation*}
where $F_n$ is the $n$th Fibonacci number. A concise description of the problem was given in terms of an equivalent process involving black and white discs and an urn (see \cite[second and third paragraphs]{griffiths2016100}). The sequence of interest here is $S(n)=-\frac{1}{2} \EE_{n}$, which is convergent as a consequence of the exponential-type behavior of the Fibonacci numbers. The Fibonacci sequence $(F_n)_{n\in\NN}$ is a D-algebraic zero of the difference polynomial $p\coloneqq \sigma^2(x)-\sigma(x)-x$, with initial values $F_0=0,\, F_1=1$. We use closure properties to find an algebraic difference equation satisfied by $(S(n))_{n\in\NN}$. From that equation, we compute an ADE satisfied by the Aitken transformation of $(S(n))_{n\in\NN},$ which we denote $(T(n))_{n\in\NN}$.
\begin{enumerate}
    \item $(u(n))_{n\in\NN}\coloneqq\left(\frac{1}{F_n\,F_{n+1}}\right)_{n\geq 2}$ satisfies the following second-order quadratic algebraic difference equation. 
    \begin{equation}
        -s\! \left(n\right) s\! \left(n+1\right)+2 s\! \left(n\right) s\! \left(n+2\right)+s\! \left(n+1\right)^{2}-s\! \left(n+1\right) s\! \left(n+2\right)=0,
    \end{equation}
    with initial terms $u(0)\coloneqq\frac{1}{2},\,u(1)\coloneqq\frac{1}{6}.$ Here, the corresponding dynamical system is similar to that of the reciprocal, with the difference that the function $f$ has the form $1/(w_1\,w_2),$ where $\sigma(w_1)=w_2$. 
    \item $(S(n))_{n\in\NN}$ satisfies a quadratic ADE of order $3$. We write the ADE as a rational recursion to highlight that the corresponding initial does not vanish at $(S(n))_{n\in\NN}$ as it is not a $C$-finite sequence.
    \begin{dmath}
        s\! \left(n+3\right)=-\frac{s\! \left(n\right) s\! \left(n+1\right)-3 s\! \left(n\right) s\! \left(n+2\right)-2 s\! \left(n+1\right)^{2}+6 s\! \left(n+1\right) s\! \left(n+2\right)-2 s\! \left(n+2\right)^{2}}{2 s\! \left(n\right)-3 s\! \left(n+1\right)+s\! \left(n+2\right)}
    \end{dmath}
    As initial terms we have $S(0)\coloneqq\frac{1}{2},\, S(1)\coloneqq\frac{2}{3}.$
    \item $(T(n))_{n\in\NN}$ is a solution of the following third-order ADE of degree $4$.
    \begin{multline}\label{eq:aitkenfib}
     s(n)^2\,s(n + 1)^2 + 6\,s(n)^2\,s(n + 1)\,s(n + 2) - 8\,s(n)^2\,s(n + 1)\,s(n + 3) + 9\,s(n)^2\,s(n + 2)^2\\
     - 24\,s(n)^2\,s(n + 2)\,s(n + 3) + 16\,s(n)^2\,s(n + 3)^2 - 12\,s(n)\,s(n + 1)^2\,s(n + 2) + 10\,s(n)\,s(n + 1)^2\,s(n + 3)\\
     - 16\,s(n)\,s(n + 1)\,s(n + 2)^2 + 44\,s(n)\,s(n + 1)\,s(n + 2)\,s(n + 3) - 24\,s(n)\,s(n + 1)\,s(n + 3)^2\\
     - 4\,s(n)\,s(n + 2)^3 + 10\,s(n)\,s(n + 2)^2\,s(n + 3) - 8\,s(n)\,s(n + 2)\,s(n + 3)^2 + 4\,s(n + 1)^3\,s(n + 2)\\
     - 4\,s(n + 1)^3\,s(n + 3) + 8\,s(n + 1)^2\,s(n + 2)^2 - 16\,s(n + 1)^2\,s(n + 2)\,s(n + 3)\\
     + 9\,s(n + 1)^2\,s(n + 3)^2 + 4\,s(n + 1)\,s(n + 2)^3 - 12\,s(n + 1)\,s(n + 2)^2\,s(n + 3)\\
     + 6\,s(n + 1)\,s(n + 2)\,s(n + 3)^2 + s(n + 2)^2\,s(n + 3)^2 = 0.        
    \end{multline}
    Despite its relatively big size, \eqref{eq:aitkenfib} enables the computation of terms of $(T(n))_{n\in\NN}$ by finding rational roots of relatively simple quadratic polynomials. For the three needed initial terms of $(T(n))_{n\in\NN}$, we need $5$ consecutive terms of $(S(n))_{n\in\NN}$. Using $S(j),j=0,\ldots,4$ we get $T(0)\coloneqq \frac{7}{9}, T(1)\coloneqq \frac{58}{75}, T(2)\coloneqq\frac{743}{960}.$ Plugging these values into \eqref{eq:aitkenfib} for $s(n), s(n+1),$ and $s(n+2)$, respectively, yields the following factored quadratic equation to compute $T(3)$.
    \begin{equation}
        \left(108300 X-83797\right) \left(10140 X-7847\right)=0.
    \end{equation}
    The desired next term is the root obtained with the positive square root of the discriminant; this is $T(3)\coloneqq\frac{7847}{10140}\approx 0.7738658777$, which is correct to $4$ decimal places with respect to the limit. This level of accuracy is achieved at the $9$th term of $(S(n))_{n\in\mathbb{N}}$.

    Table 1 in \cite{griffiths2016100} presents a comparative analysis of the convergence rates for $(S(n))_n$ and $(T(n))_n$, demonstrating an example of acceleration achievable through Aitken's delta-squared process. Finding an ADE satisfied by the accelerating sequence allows for the computation of its terms independently, without needing to refer back to the original sequence. \QEDA
\end{enumerate}
\end{example}

\section{D-algebraic rational recursions}

In this section, we focus on two subclasses of D-algebraic sequences, namely holonomic and $C^2$-finite sequences. We show that generic sequences from both classes are D-algebraic zeros of rationalizing difference polynomials, defined as difference polynomials that are linear in their highest shift terms (see \Cref{def:diffpoly}). As in the previous section, we concentrate on simple D-algebraicity, which corresponds to sequences entirely defined by an irreducible difference polynomial. By generic sequences in these subclasses, we want to exclude any behavior arising, for instance, from repeated values within the sequence terms, which may require polymorphic D-algebraicity where one of the irreducible difference polynomials is rationalizing. The power series coefficients of $\arctan$ (which exhibit repeated zeros) provide one such example, as does the D-algebraic continuation of Bernoulli numbers discussed in \Cref{ex:bernoullicont}. 

Application-wise, we note that rational recursions establish a natural connection to automata construction with rational updates in theoretical computer science \cite{AlurDDRY13,benedikt2017polynomial}.

\subsection{Holonomic sequences}\label{subsec:holorec}

Holonomic sequences are solutions to linear recurrence equations with polynomial coefficients in the index variable. These sequences are ubiquitous in the sciences. One reason may be that they share similar properties with their generating functions. Some interesting applications can be found in \cite{bostan2024algebraic,kauers2011concrete}.

In \cite{teguia2024rational}, Worrell and the author proposed an algorithm to convert any holonomic equation into a rationalizing algebraic difference equation. We revisit this result and complete its proof.

For this part, we replace the field $\KK$ with $\KK(n)$ and work on the specialized ring of difference polynomials $\DD_s(n)\coloneqq \KK(n)[\sigma^{\infty}(s(n))]$. This enables us to introduce the index variable $n$ in our difference polynomials. However, the aim is to derive a difference polynomial $p\in\DD_s=\KK[\sigma^{\infty}(s(n))]$ where the term $s(n+\ord(p))$ appears linearly.

We remind that our focus is on D-algebraic zeros, in the sense that for any rationalizing difference polynomial $p$ obtained from a holonomic difference polynomial, our interest is on sequence solutions $(s(n))_n$ such that the set $\{n\in\NN,\, \Initp(s(n))=0\}$ is finite. For holonomic difference polynomials, this is clear, so all their solutions are D-algebraic in $\DD_s(n).$ However, it requires some more work to construct a rationalizing difference polynomial in $\DD_s$ from a holonomic sequence. We provide some ideas of sequences that we would like to avoid in the example below.

\begin{example}\label{ex:arctanexpl} The sequence of power series coefficients of $\arctan(z)$, say $(a(n))_n$ satisfies the holonomic equation
\begin{equation}\label{eq:holeqarct}
    n s\! \left(n\right)+\left(n+2\right) s\! \left(n+2\right)=0,
\end{equation}
with initial terms $a(0)=0, a(1)=1$ (see \cite{BTWKsymbconv} and references therein for further details). Using the algorithm from \cite{teguia2024rational} to convert \eqref{eq:holeqarct} into a rationalizing ADE (or rational recursion) yields
\begin{equation}\label{eq:ratrecarct}
    s\! \left(n+3\right)=-\frac{s\! \left(n+1\right) \left(-s\! \left(n+2\right)+s\! \left(n\right)\right)}{s\! \left(n+2\right)+3 s\! \left(n\right)}.
\end{equation}
However, $(a(n))_n$ is not a D-algebraic solution of this ADE since all its terms of even indices are zero, forcing its initial to be zero infinitely many times. Nevertheless, starting a sequence solution of \eqref{eq:holeqarct} at $n\geq 1$ with two nonzero initial terms enables to define a holonomic solution of \eqref{eq:holeqarct} that is a D-algebraic solution of \eqref{eq:ratrecarct}. 

On the other hand, it is relatively simple to deduce a difference polynomial satisfied by $(a(n))_n$ by considering its sequences of zero terms and nonzero terms separately. For the zero terms, given $n\in\NN$, and $s(n)=0$, we want $s(n+2)=0$, so the equation is simply $s(n+2).$ For the nonzero terms, we convert the $2$-fold equation \eqref{eq:holeqarct} in to a $1$-fold equation using the change of variables $n\rightarrow 2n+1,$ $s(2n+1)\rightarrow s(n)$ (see \cite{BTWKsymbconv}). We obtain $(2n+1)s(n)+(2n+3)s(n+1)=0$, which we convert into the rationalizing difference polynomial
\begin{equation}
    (s(n+1)+2s(n))s(n+2)+s(n+1)s(n).
\end{equation}
Now we interpret $s(n),s(n+1)$ and $s(n+2)$ as $s(n),s(n+2)$ and $s(n+4)$, respectively. This enables taking the index gap with the zero terms into account. Hence $(a(n))_n$ is a D-algebraic zero of
\begin{equation}
    s(n+2)\,\left((s(n+2)+2s(n))s(n+4)+s(n+2)s(n)+s(n+1)\right),
\end{equation}
which can also be written as $\sigma^2(x)\left((\sigma^2(x)+2x)\sigma^4(x)+\sigma^2(x)x+\sigma(x)\right)\in\Dsigma$. The term $\sigma(x)$ (or $s(n+1)$) is added to ensure that only one factor vanishes at every $n\in\NN$. One can also obtain this difference polynomial using the ``D-algebraic continuation method'' exploited in \Cref{ex:bernoullicont}. \QEDA
\end{example}

\begin{definition} The holonomic degree of a difference polynomial $p\in\DD_s(n)$ of degree $1$ is the degree of $p$ viewed as a univariate polynomial in $n$. 
\end{definition}
When working with holonomic equations or sequences, we often use the word ``degree'' to refer to the holonomic degree. The order of the minimal holonomic difference polynomial satisfied by a sequence is the order of that holonomic sequence.

In \Cref{ex:arctanexpl}, the sequence of nonzero terms is a D-algebraic zero of a second-order rationalizing difference polynomial. This indicates that these terms can be entirely described by a rational recursion, despite the zero terms at odd indices. To avoid these case-by-case treatments of zeros of holonomic difference polynomials, we consider the concept of {\it almost all} for sequences defined by them. They have the particularity of using the full basis of solutions, unlike our previous example, where one term of the basis is killed by $0$ in the initial terms of the sequence.

\begin{definition}\label{def:almostall} Let $p\in\DD_s(n)$ be a holonomic difference polynomial of order $l_1,$ and $q\in\DD_s$ of order $l_2$, $l_2\geq l_1$. Let $\tilde{p}=\sigma^{N+1}(p),$ where $N$ is the maximum nonnegative integer root of the polynomial coefficients in $p.$ Let $(\epsilon_p(n))_{n>N}$ be the solution of $\tilde{p}$ with symbolic initial terms
\[\epsilon_p(j+N)=X_j\in\{X_1,\ldots,X_{l_1}\},\]
where the $X_j'$s are viewed as distinct polynomial variables not belonging to $\KK$. We say that almost all zeros of $p$ are D-algebraic zeros of $q$ if $(\epsilon_p(n))_{n>N}$ is a generic zero of $q$.
\end{definition}

Note that in \Cref{def:almostall}, $\epsilon_p(l_1+1),\ldots,\epsilon_p(l_2-1)$ are computed using $p$; and the generic zeros and D-algebraic zeros of $q$ coincide since the polynomials are free of $X_j$'s. In \Cref{ex:arctanexpl}, we have $N=0$, and the corresponding $(\epsilon(n))_{n>0}$ is given by
\begin{equation}
    \left(X_{1},X_{2},-\frac{X_{1}}{3},-\frac{X_{2}}{2},\frac{X_{1}}{5},\frac{X_{2}}{3},-\frac{X_{1}}{7},-\frac{X_{2}}{4},\frac{X_{1}}{9},\frac{X_{2}}{5},-\frac{X_{1}}{11},-\frac{X_{2}}{6},\frac{X_{1}}{13},\ldots\right).
\end{equation}
One verifies that this sequence can also be generated with \eqref{eq:ratrecarct} using the first three terms. The purpose of \Cref{def:almostall} is twofold: to highlight the dependence on the first terms and help demonstrate the non-triviality of the main result from \cite{teguia2024rational}.

\begin{definition}\label{def:almostall2} We say that almost every holonomic sequence is a D-algebraic zero of a family of difference polynomials $\mathcal{F}$, if for any holonomic difference polynomial $p,$ almost all the zeros $p$ are D-algebraic zeros of some $q\in\mathcal{F}.$
\end{definition}

Besides excluding particular solutions of holonomic equations, \Cref{def:almostall} and \Cref{def:almostall2} also neglect irrelevant difference polynomials satisfied by holonomic sequences. For instance, if a holonomic sequence satisfies a difference polynomial $p\in\DD_s$ of order $l$, then it also satisfies the rationalizing difference polynomial $q=\sigma^{l+2}(x)\sigma(p)+p$. However, the corresponding symbolic solution $(\epsilon_p(n))_{n>N}$ is not a generic zero of $q$, as it is a zero of $\Init_q$. 

\begin{theorem}\label{th:holoisratrec} Almost every holonomic sequence of order $l$ and degree $d$ is a D-algebraic zero of a rationalizing difference polynomial of order at most $l+d$ in $\DD_s$.   
\end{theorem}
\begin{proof} Let $p\in\DD_s(n)$ be holonomic of degree $d$ and order $l$. If $d=0$, then we are done. For the rest of the proof, we assume that $d>0$. This assumption implies that $(s(n))_n$ (as a sequence) does not vanish any holonomic difference polynomial of order $\leq l$ and degree $\leq d$. In particular, no constant coefficient linear recurrence of order $\leq l$ is satisfied by $(s(n))_n$.

We look at $p$ and its shifts as polynomials in $n$ such that
\begin{equation}\label{eq:shiftgamma}
    \sigma^j(p) = \sum_{k=0}^d \gamma_{j,k}\,n^k,\,\, j=0,\ldots,d,
\end{equation}
where $\gamma_{j,k}\in\KK[s(n+j),\ldots,s(n+j+l)]$. We consider the associated equations for $j<d$ and write them as follows
\begin{equation}\label{eq:nksys}
   \sum_{k=1}^d \gamma_{j,k}\,n^k = -\gamma_{j,0},\,\, j=0,\ldots,d-1.
\end{equation}
This is a linear system of $d$ equations in $(n,n^2,\ldots,n^d)^T$. We claim that the matrix of the system has a generic full rank due to the shifts involved in the $\gamma_{j,k}$'s. Indeed, each $\gamma_{j,k}$ has the following form:
\begin{equation}\label{eq:gammacsdef}
    \gamma_{j,k} \coloneqq \sum_{i=0}^{l} c_{j,i,k}\,s(n+j+i) = C_{j,k}\cdot S_{l+j},
\end{equation}
where $c_{j,i,k}$ is the constant coefficient of $n^k$ in the polynomial coefficient of $s(n+j+i)$ in $\sigma^j(p)$, $i=0,\ldots,l$, $j=0,\ldots,d-1$; $S_{l+j}=(s(n+j),\ldots,s(n+j+l))^T$ and $C_{j,k}=(c_{j,0,k},c_{j,1,k},\ldots,c_{j,l,k})$. Note that all $c_{j,i,k}$'s are linear combinations of the $c_{0,i,k}$'s, and $c_{j,i,d}=c_{0,i,d}$ for all $j=0,\ldots,d-1$.

Let $\Gamma_0,\ldots,\Gamma_{d-1}$ be the rows of the matrix from \eqref{eq:nksys}. Observe that 
\begin{equation}\label{eq:Gammacsdef}
    \Gamma_j = (\gamma_{j,1},\ldots,\gamma_{j,d}) = (C_{j,1}\cdot S_{l+j},\ldots, C_{j,d}\cdot S_{l+j}).
\end{equation}
The components of the $\Gamma_j$'s are all nonzero as they all yield $l$th-order linear recurrence equations with constant coefficients. 

Let $\lambda_0,\lambda_1,\ldots,\lambda_{d-1}\in \KK(s(n),\ldots,s(n+l+d-1))$ such that $\sum_{j=0}^{d-1} \lambda_j\,\Gamma_j=0$. The latter is a homogeneous linear system of linearly independent equations in the $\lambda$'s and can only have the trivial solution. 

To see that, notice that for $0\leq j_1\neq j_2\leq d-1$, $\gamma_{j_1,k}$ and $\gamma_{j_2,k}$ can be seen as two independent variables because $S_{l+j_1}$ and $S_{l+j_2}$ are two linearly independent vectors over $\KK(s(n),\ldots,s(n+l+d-1))$. 

Furthermore, one can always find indices $1\leq k_1\neq k_2\leq d$ such that the constant coefficients $c_{j,i,k_1}$ and $c_{j,i,k_2}$ (appearing in front of $s(n+j+i)$ for some $i$) are distinct. This distinction is vital; without it, the holonomic equation would simplify to a $C$-finite equation of the same order, contradicting our hypothesis. This specific characteristic implies that $\gamma_{j,k_1}$ and $\gamma_{j,k_2}$ are also effectively distinct variables, given that they generically represent two linearly independent homogeneous linear forms in $\KK[s(n+j),\ldots,s(n+l+j)]$.

Thus, the homogeneous system in the $\lambda$'s can be represented by the following matrix:
\begin{equation}\label{eq:detgamma}
    M_{\gamma}\coloneqq 
    \begin{bmatrix}
        \gamma_{0,1}   & \gamma_{0,2}   & \ldots & \gamma_{0,d}  \\
        \gamma_{1,1}   & \gamma_{1,2}   & \ldots & \gamma_{1,d}  \\
        \vdots         & \vdots         & \ldots & \vdots        \\
        \gamma_{d-1,1} & \gamma_{d-1,2} & \ldots & \gamma_{d-1,d}\\
    \end{bmatrix}.
\end{equation}
This matrix $M_{\gamma}$ can be regarded as a square matrix where its $d^2$ components are distinct variable components, such as $(x_1,x_2,\ldots,x_{d^2})$. Alternatively, and perhaps more naturally, $M_{\gamma}$ can be viewed as a square matrix whose components are homogeneous linear forms with generic coefficients. In this perspective, the components within the same column share the same underlying set of variables, while those of two consecutive columns exhibit an overlapping structure.

Therefore, the determinant of $M_{\gamma}$ is nonzero, implying that $\lambda_0=\lambda_1=\cdots=\lambda_{d-1}=0$. It is essential to note that this analysis relies on purely symbolic computation involving shifts of $s(n)$. The value of $|M_{\gamma}|$ for a particular numerical value of $n$ is not the primary concern; rather, the focus is on establishing that $|M_{\gamma}|$ is not the zero-polynomial in $\KK[s(n),\ldots,s(n+l+d-1)]$. 

Hence \eqref{eq:nksys} is non-singular, which implies that for all positive integers $k\leq d$, $n^k$ belongs to the field $\KK(s(n),\ldots,s(n+d+l-1))$. 

Now, let us consider $\sigma^d(p)$, viewed as a holonomic difference polynomial. We must explain why the $d$th shift of the leading polynomial coefficient of $p$ does not vanish when substituting the expressions of the $n^k$'s. Indeed, if it were to vanish, this would imply that the equations of the linear system solved to obtain the $n^k$ are linearly dependent with the polynomial coefficient of $s(n+d+l)$. This, however, is impossible because all coefficients of the linear system in \eqref{eq:nksys} are homogeneous linear forms in $\KK[s(n),\ldots,s(n+l+d-1)]$, each involving at least two `variables'. Such homogeneous linear forms cannot linearly combine to produce the constant coefficients that are present in the polynomial coefficient of $s(n+d+l)$ in $\sigma^d(p)$.

Furthermore, the symbolic zero $(\epsilon_p(j))_{j>N}$ cannot vanish the leading polynomial coefficient of $\sigma^d(p)$ since writing all its terms in terms of $X_1,X_2,\ldots,X_l$ and values of $n=j$, would simply evaluate that leading polynomial coefficient at $j$. And since $j>N,$ where $N$ is the maximum nonnegative integer roots of the polynomial coefficients of $p$, we can be sure that the leading polynomial coefficient is nonzero for all $j>N$.

Finally, since all the $n^k$'s are free of $s(n+l+d)$, the resulting difference polynomial is rationalizing with order at most $l+d$ in $\DD_s$, and almost every zero of $p$ is one of its D-algebraic zeros.
\end{proof}

Our implementation of the algorithm from the proof of \Cref{th:holoisratrec} is now part of the \texttt{DalgSeq} subpackage of \texttt{NLDE}. The corresponding procedure is \texttt{HoloToSimpleRatrec}.

\begin{example}[Generating Somos-like sequences \cite{malouf1992integer,ekhad2014generate}] A Somos-like sequence is an integral sequence defined with a rational recursion. Using \Cref{th:holoisratrec}, one can generate a Somos-like sequence as follows:
\begin{enumerate}
    \item Take a holonomic equation and choose integral initial values such that all the following terms are also integers. A natural choice is to take an equation of the form
    \begin{equation}\label{eq:holosomos}
        s(n+k+1) = P_0(n)s(n)+\cdots+P_k(n)s(n+k),
    \end{equation}
    $P_i(n)\in\KK[n],i=0,\ldots,k$, with any set of $k+1$ integers for the initial values.
    \item Then use the algorithm from the proof of \Cref{th:holoisratrec} to convert \eqref{eq:holosomos} into a rationalizing difference polynomial.
\end{enumerate}
    Concretely, let us take
    \begin{equation}\label{ex:somos}
        s(n+3)=n s(n)+(n+1) s(n+1)+(n+2) s(n+2).
    \end{equation}
    We choose the initial values $s(0)=s(1)=s(2)=1$. This is \href{https://oeis.org/A122752}{A122752} from \cite{sloane2003line}. We have $s(3)=3$. Thus, the sequence defined by \eqref{ex:somos} and these initial values is an integer sequence. Our procedure \texttt{HoloToSimpleRatrec} produces the following rational recursion from \eqref{ex:somos}:
    \begin{multline}\label{ex:ratrec}
        s(n+4)=\frac{1}{{s(n)+s(n+1)+s(n+2)}}\Big(s(n) s(n+1)+2 s(n) s(n+2)+3 s(n) s(n+3)\\
        +3 s(n+1) s(n+3)+2 s(n+2) s(n+3)+s(n+3)^{2}\Big).
    \end{multline} \QEDA
\end{example}

\begin{example}\label{ex:catholotorec} Catalan numbers are defined by the formula $C(n)\coloneqq\frac{1}{n+1}\binom{2n}{n}$, and satisfy the equation $(n+2)s(n+1)-(4 n+2) s(n)=0$, which is holonomic. Using \texttt{HoloToSimpleRatrec}, we convert that equation into the rational recursion
\begin{equation}
    s(n+2)=\frac{2 s(n+1) \left(8 s(n)+s(n+1)\right)}{10 s(n)-s(n+1)}.
\end{equation}
The above recursion appeared on the OEIS website for \href{https://oeis.org/A000108}{A000108} in 2006.  

A noteworthy observation emerges from the geometrical view of \Cref{th:holoisratrec}. For the present example, with $\KK=\RR$, let us consider the surface defined by 
\begin{equation}
    H_C\coloneqq \left\lbrace (x,y,z),\,\RR^3\,\colon z\,(x+2)-y\,(4 x+2)=0\right\rbrace.
\end{equation}
This encodes the holonomic equation of Catalan numbers. On this surface, Catalan numbers are the points $P_n \coloneqq (n, C(n), C(n+1)),n\in\NN$. For the D-algebraic representation, we consider 
\begin{equation}
    H_C'\coloneqq \left\lbrace (x,y,z),\,\RR^3\,\colon z\,(10 x - y) - (8 x + y)=0\right\rbrace,
\end{equation}
where Catalan numbers are the points $P_n'\coloneqq (C(n), C(n+1), C(n+2)), n\in\NN$. It turns out that $H_C$ and $H_C'$ intersect on a curve surrounded by the points $P_n$ and $P_n'$, with $P_1=P_0'$ on it. The projection of the  algebraic variety thus defined onto the $xy$-plane is given by the curve
\begin{equation}
    \mathcal{C}_C\coloneqq \left\lbrace (x,y),\,\RR^2\,\colon y\,(x + 1) - 2 x (2 x - 1)=0\right\rbrace.
\end{equation}
\Cref{fig:catalanspace} illustrates this geometry from the first octant of the $xyz$-space.
\begin{figure}[H]
        \centering
        \includegraphics[scale = 0.20]{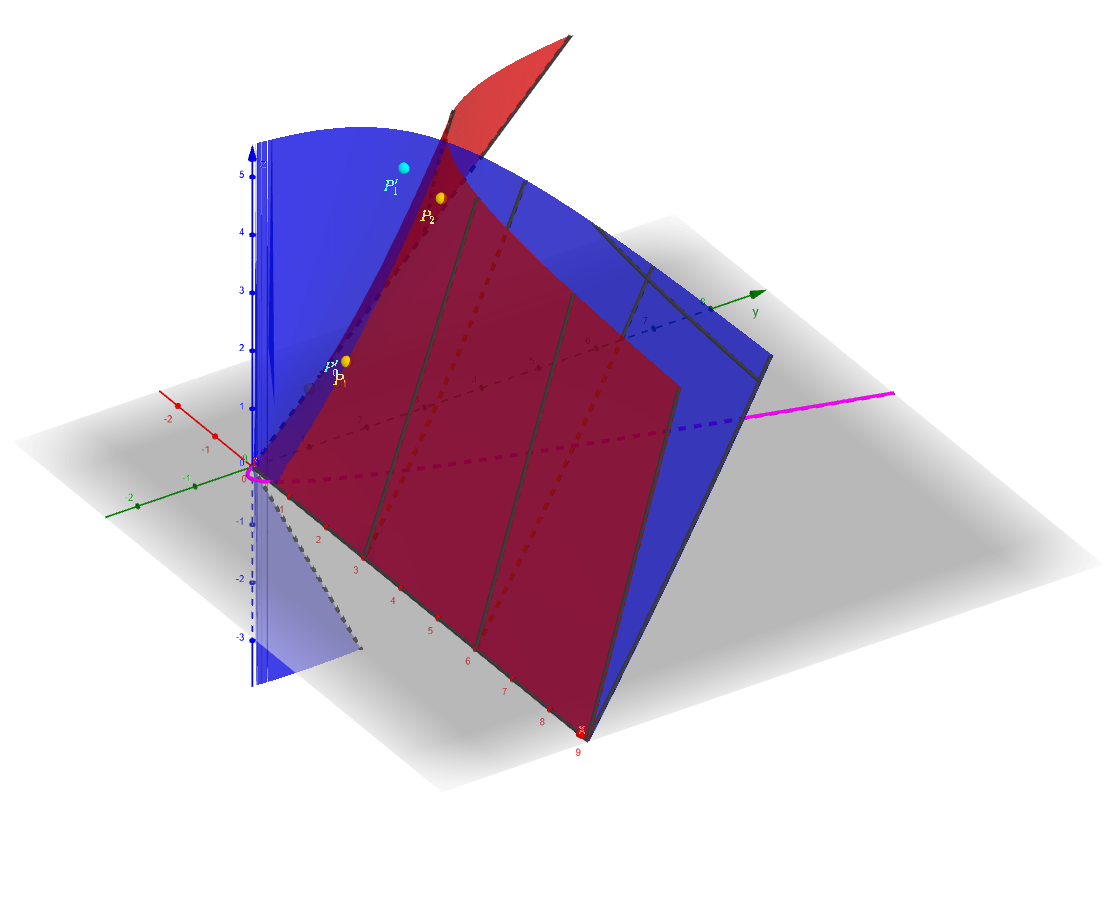}
        \captionsetup{justification=centering}
        \caption{A geometric view of \Cref{th:holoisratrec} for Catalan numbers.\\
        $H_C$ is in red, $H_C'$ in blue, and $\mathcal{C}_C$ is in magenta.}
        \label{fig:catalanspace}
    \end{figure}
Further study is needed to elucidate the connection between such varieties and their related holonomic sequences.\QEDA
\end{example}

\subsection{$C^2$-Finite sequences}

A $C$-finite sequence solves a linear recurrence equation with constant coefficients. A $C^2$-finite sequence is a solution to a linear recurrence equation with $C$-finite term coefficients \cite{jimenez2023extension}. As in the differential case with $D^2$-finite functions, $C^2$-finite sequences are D-algebraic. A ``slow'' Groebner bases-based implementation for computing ADEs from $C^2$-finite sequences is provided in our package under the name \texttt{CCfiniteToDalg}. The approach is similar to that of $D^2$-finite functions as described in \cite{teguia2023operations}. However, in this subsection, we establish a result similar to \Cref{th:holoisratrec} for this class of sequences. This result can also be seen as a consequence of the result from \cite{clemente2023rational}.

As in the previous subsection, the concept of ``almost every'' can also be defined. However, for this part, we only concentrate on the algebraic manipulations. Before stating the main result, let us illustrate the idea on a generic first-order $C^2$-finite sequence with second-order $C$-finite term coefficients.

Let $(s(n))_n$ be a first-order $C^2$-finite sequence with $C$-finite coefficients $(c_1(n))_n$ and $(c_0(n))_n$, such that
\begin{align}
    c_1(n)\,s(n+1)&+c_0(n)\,s(n)=0, \label{eq:ccde}\\
    \alpha_{1,1}\,c_1(n+1)+\alpha_{1,0}\,c_1(n)=c_1(n+2),&\quad\, \alpha_{0,1}\,c_0(n+1)+\alpha_{0,0}\,c_{0}(n)=c_0(n+2) \label{eq:cde}
\end{align}
where $\alpha_{i,j}\in\KK,\, i,j\in\{0,1\}.$ In what follows, we assume that denominators in rational expressions are nonzero for some integers $n$ within the domain of interest. 

Our idea is to eliminate the $c_i$'s from the shifts of \eqref{eq:ccde} by incremental substitution. We start by eliminating $c_1(n)$ and its shifts. From \eqref{eq:ccde} and its first shift, we have
\begin{equation}
 c_{1}\! \left(n\right)=-\frac{c_{0}\! \left(n\right) s\! \left(n\right)}{s\! \left(n+1\right)},\quad\,   c_{1}\! \left(n+1\right)=-\frac{c_{0}\! \left(n+1\right) s\! \left(n+1\right)}{s\! \left(n+2\right)}.\label{eq:elimc1}
\end{equation}
We now proceed with the elimination of $c_0(n)$. We consider the second shift of \eqref{eq:ccde} with the corresponding substitutions using \eqref{eq:cde} and \eqref{eq:elimc1}. We obtain
\begin{equation}
c_{0}\! \left(n\right)=-\frac{s\! \left(n+1\right) \left(s\! \left(n+2\right)^{2} \alpha_{0,1}-s\! \left(n+1\right) s\! \left(n+3\right) \alpha_{1,1}\right) c_{0}\! \left(n+1\right)}{s\! \left(n+2\right) \left(s\! \left(n+1\right) s\! \left(n+2\right) \alpha_{0,0}-s\! \left(n\right) s\! \left(n+3\right) \alpha_{1,0}\right)}.
\end{equation}
Given that all shifts of $C$-finite coefficients can now be written as rational multiples of $c_0(n+1)$, it follows that their substitution in the third shift of $\eqref{eq:ccde}$ yields the product of $c_0(n+1)$ and a rational expression in the shifts of $s(n)$, in which $s(n+4)$ appears linearly. Hence, after canceling $c_0(n+1)$ we deduce the desired recursion:
\[s\! \left(n+4\right)=\frac{\mathsf{N}}{\mathsf{D}}\coloneqq\frac{\mathsf{N}(s(n),\ldots,s(n+3))}{\mathsf{D}(s(n),\ldots,s(n+3))},\]
where
\begin{small}
\begin{dmath}
\mathsf{N}=-s\! \left(n+3\right) \left(-s\! \left(n+2\right) s\! \left(n\right) s\! \left(n+3\right) \alpha_{0,1}^{2} \alpha_{1,0}+s\! \left(n+1\right)^{2} s\! \left(n+3\right) \alpha_{0,0} \alpha_{0,1} \alpha_{1,1}+s\! \left(n+2\right)^{2} s\! \left(n+1\right) \alpha_{0,0}^{2}-s\! \left(n+2\right) s\! \left(n\right) s\! \left(n+3\right) \alpha_{0,0} \alpha_{1,0}\right),\label{eq:firstordnum}
\end{dmath} 
\begin{dmath}
\mathsf{D}=s\! \left(n+2\right)^{2} s\! \left(n\right) \alpha_{0,1} \alpha_{1,0} \alpha_{1,1}-s\! \left(n+2\right) s\! \left(n+1\right)^{2} \alpha_{0,0} \alpha_{1,1}^{2}-s\! \left(n+2\right) s\! \left(n+1\right)^{2} \alpha_{0,0} \alpha_{1,0}+s\! \left(n+1\right) s\! \left(n\right) s\! \left(n+3\right) \alpha_{1,0}^{2}.\label{eq:firstorddenom}
\end{dmath}
\end{small}

This suffices to prove our following proposition.

\begin{proposition} Almost every first-order $C^2$-finite sequence with second-order $C$-finite term coefficients is a D-algebraic zero of a rationalizing difference polynomial of order at most $5$.
\end{proposition} 

Note that if some of the coefficients in the $C^2$-finite equation are nonzero constants, then in the third shift of that equation, not all $C$-finite coefficients (including constants) would be expressible as a rational function in the shifts of $s(n)$ multiplied by some $c_i(n+1),\, i\in\{0,1\}$. Therefore, one more shift would be required to find the rational recursion. This explains why the bound is $5$ and not $4$.

\begin{example}\label{ex:firstordccfinite} We take $(c_1(n))_n$ as a solution of the Fibonacci recurrence equation, and $(c_0(n))_n$ defined with the recursion $c_0(n+2)=2\,c_0(n+1)+3\,c_0(n)$. We have $\alpha_{1,1}=\alpha_{1,0}=1,$ $\alpha_{0,1}=2,$ and $\alpha_{0,0}=3$. Using \eqref{eq:firstordnum} and \eqref{eq:firstorddenom}, we deduce the rational recursion
\begin{equation}
s\! \left(n+4\right)=-\frac{s\! \left(n+3\right) \left(6 s\! \left(n+3\right) s\! \left(n+1\right)^{2}-7 s\! \left(n+3\right) s\! \left(n\right) s\! \left(n+2\right)+9 s\! \left(n+2\right)^{2} s\! \left(n+1\right)\right)}{s\! \left(n+3\right) s\! \left(n+1\right) s\! \left(n\right)+2 s\! \left(n\right) s\! \left(n+2\right)^{2}-6 s\! \left(n+1\right)^{2} s\! \left(n+2\right)}.
\end{equation}
The Gr\"obner bases method implemented in \texttt{CCfiniteToDalg} returns this same equation in about $85$ seconds in CPU time.\QEDA
\end{example}

Let us now present the general algorithm. To simplify notations, we assume that the input data can be updated within the algorithm.

\begin{algorithm}[ht]\caption{$C^2$-finite to D-algebraic rational recursion}\label{algo:Algo2}
    \begin{algorithmic} 
    \\ \Require \item 
          \begin{itemize}
              \item $C^2$-finite equation $(p):c_0(n)s(n)+\cdots+c_l(n)s(n+l)=0,\,l>0$.
              \item $C$-finite equations $(q_j):c_j(n+r_j)=\alpha_{i,0}c_j(n)+\cdots+\alpha_{j,r_j-1}c(n+r_j-1),$ $j\in\{0,\ldots,l\}.$
          \end{itemize}
    \Ensure A rationalizing difference polynomial of order at most $l+\sum_{j=0}^l r_j$ in $\DD_s$.
    \begin{enumerate}
    \item Set $r_l\coloneqq r_l-1$, unless some $r_j=0$ and $c_j(n)\neq 0$.
    \item For each $j$, $0\leq j \leq l$ do:
        \begin{itemize}
            \item[2.1.] For each $k$, $0\leq k \leq r_j-1,$ do:
            \begin{itemize}
                \item[2.1.1.] $c_j(n+k)\coloneqq\text{solve}((p),c_j(n+k)).$ /\!/root of a univariate linear polynomial.
                \item[2.1.2.] $(p)\coloneqq \sigma((p))$;
                \item[2.1.3.] Update $(p)$ and the $(q_i),\, 0\leq i \leq l,$ by substituting $c_j(n+k)$ as solved.
            \end{itemize}
        \end{itemize}
    \item Let $r\coloneqq l + \sum_{j=0}^l r_j$.
    \item Return ``$s(n+r)=\text{solve}((p),s(n+r))$'' or its associated difference polynomial.


    \end{enumerate}	
    \end{algorithmic}
\end{algorithm}

\begin{theorem}\label{th:ccflhs} Almost every $C^2$-finite sequence of order $l$ with $C$-finite term coefficients of order $r_j,j=~0,\ldots,l$, is a D-algebraic zero of a rationalizing difference polynomial of order at most $l+\sum_{j=0}^lr_j.$
\end{theorem}
\begin{proof}
    This is a consequence of the correctness of \Cref{algo:Algo2}, which might be detailed as follows.
    \begin{itemize}
        \item Each $c_j(n+k),\,0\leq j\leq l,$ and $0\leq k<r_j$, is eliminated in step 2.1.1. of \Cref{algo:Algo2}.
        \item This elimination enables the replacement of all appearances of $c_j(n+k)$ from the $(q_j)$, $(p)$, and all $c_i(n+k_i)$ that are already eliminated.
        \item Therefore, when the loop in step $2$ ends, each $c_j(n+k)$'s on the right hand side of each $C$-finite equation $(q_j)$ is a rational expression of shifts of $s(n)$, unless all $r_j\geq 1$, in which case $c_l(n+r_l-1)$ does not need to be expressed as a rational expression of shifts of $s(n)$. 
        \item The number of shifts applied to $(p)$ is exactly $r=l+\sum_{j=0}^lr_j$, with $r_l$ potentially updated in step $1$. In this shift of $(p)$, one of the factors is a rationalizing difference polynomial in which $s(n+r)$ appears linearly.
    \end{itemize}
\end{proof}

\begin{remark}\item 
\begin{itemize}
    \item As explained in the first-order case, in \Cref{algo:Algo2}, when some $C$-finite term coefficients are nonzero constants, i.e, $r_j=0$ and $c_j(n)\neq 0$, all $c_j(n+k),\, 0\leq k< r_j$ have to be expressed as rational expressions of shifts of $s(n)$, and in this case the order of the output is exactly $l+\sum_{j=0}^lr_j$, where $r_l$ is not updated in step 1. 
    \item The arrangement of the $C$-finite term coefficients in the input is irrelevant for the algorithm: Interchanging the indices $i$ and $j>i$ to eliminate $c_j$ before $c_i$ in the list of $C$-finite coefficients does not affect the correctness of the algorithm.
    \item \Cref{th:ccflhs} is also another proof of \Cref{th:holoisratrec} since every holonomic sequence  is $C^2$-finite as every polynomial satisfies a $C$-finite equation.
    \item The elimination procedure in Algorithm 2 is inherently similar to the linear algebra computations detailed in \cite[Section 4]{jimenez2020some}. Specifically, an analysis of the resultant approach (Theorem 28 in that paper) reveals that this method invariably yields a rationalizing difference polynomial when applied to $C^2$-finite equations. This finding was later explained by Jimenez-Pastor (private communication), after we had presented \Cref{th:ccflhs} to him.
\end{itemize}
\end{remark}

We implemented \Cref{algo:Algo2} as \texttt{CCfiniteToSimpleRatrec}. The latter is much more efficient than \texttt{CCfiniteToDalg}. The crucial advantage of \Cref{algo:Algo2} is that it does not use Gr\"{o}bner bases. We also observed that when \texttt{CCfiniteToDalg} completes its computations in a reasonable time, its output corresponds with that of \texttt{CCfiniteToSimpleRatrec}.

The approach used in \texttt{CCfiniteToDalg} requires that $c_l(n)\neq 0$ for sufficiently large $n$. This condition is satisfied when $(c_l(n))_n$ is \textit{non-degenerate} \cite{berstel1976deux}. This means that subsequences of $(c_l(n))_n$ indexed by arithmetic progressions are not $C$-finite of orders lower than $r_l$ \cite{berstel1976deux,methfessel2000zeros}. In the general case, however, the proper definition of a $C^2$-finite sequence is subject to answering the \textit{Skolem problem} \cite{ouaknine2012decision} for the leading $C$-finite term coefficient. Note that this constraint highly depends on the choice of initial values: the sequence $((-1)^n-1)_n$ satisfies the recurrence $s(n+2)-s(n)=0$ and has infinitely many zeros; however, the sequence $(2\,(-1)^n-1)_n$ satisfies the same recurrence equation but has no zero. In the following example, we use \texttt{CCfiniteToSimpleRatrec} to compute a rational recursion satisfied by a $C^2$-finite sequence for which two not-necessarily-equal solutions of $s(n+2)=s(n)$ appear within the coefficients of its equation.

\begin{example} Consider a $C^2$-finite sequence $(s(n))_n$ solution of 
\begin{equation}
    u(n)\,s(n) + 2\,s(n + 1) + v(n)\,s(n + 2) = 0,
\end{equation}
where $(u(n))_n$ and $(v(n))_n$ are two not-necessarily identical solutions of $s(n+2)-s(n)=0$. This includes the result of \cite[Example 4.1]{jimenez2023extension}. Executing the following code yields the result in $0.031$ second in CPU time.
\begin{lstlisting}
> with(NLDE:-DalgSeq); #loading the subpackage DalgSeq from NLDE
> CCfiniteToSimpleRatrec(u(n)*s(n) + 2*s(n + 1) + v(n)*s(n + 2) = 0, s(n), 
[u(n + 2) - u(n) = 0, v(n + 2) - v(n) = 0], [u(n), v(n)])
\end{lstlisting}
\begin{small}
\begin{equation*}
s(n+6)=\frac{ s(n+5) s(n+4) s(n)-s(n+5) s(n+2)^{2}-s(n+4)^{2} s(n+1)+s(n+4) s(n+3) s(n+2)}{s(n+3) s(n)-s(n+1) s(n+2)}.
\end{equation*}
\end{small}
Our \texttt{CCfiniteToDalg} command gives the same result (in the difference polynomial form) in $89.438$ seconds in CPU time. We provide the syntax of the code without output below to highlight the differences in syntax between the two procedures.
\begin{lstlisting}
> CCfiniteToDalg(u*s(n) + 2*s(n + 1) + v*s(n + 2) = 0, s(n), 
[u(n + 2) - u(n) = 0, v(n + 2) - v(n) = 0], [u(n), v(n)]):
\end{lstlisting}
Here, the $C$-finite term coefficients do not appear in the $C^2$-finite equation with the index variable. A full documentation of these commands is provided at \url{https://t3gu1a.github.io/NLDEdoc/DalgSeq-Commands-and-Examples.html}.
 \QEDA
\end{example}

\section{Subsequences}

Recall that a subsequence of a sequence $(s(n))_n$ is a sequence $(t(n))_n$ such that $t(n)=s(\varphi(n))$, where $\varphi$ is a strictly increasing embedding of $\NN$. Thus, to find a difference polynomial $p$ that vanishes at $(t(n))_n$, we need an endomorphism $\tilde{\sigma}$ that sends $s(\varphi(n))$ to $s(\varphi(n+1))$, i.e., the minimal gap between the orders of two difference monomials in $p$ is $\delta_{\tilde{\sigma}} = \min_{n\in\NN} \{\varphi(n+1)-\varphi(n)\}$ and not the usual $n+1-n=1$. This section focuses on the case where $\varphi(n)=dn$. This is the $n$th term formula of an arithmetic progression of common difference $d$ with initial term $0$. What makes this case rather natural is that all order gaps are multiples of the common difference $d$, and here we have $\delta_{\tilde{\sigma}}=d$. This implies that, assuming $x$ occurs in $p$, we may write $p$ as follows:
\begin{equation}
    p(x,\sigma^d(x),\sigma^{2d}(x),\sigma^{3d}(x),\ldots,\sigma^{dk}(x),\ldots).
\end{equation}
The desired endomorphism is thus defined by $\tilde{\sigma}=\sigma^d$. The core idea of our next theorem is to exploit the flexibility in choosing the endomorphism during the construction of the dynamical system $(\mathcal{M}_f)$ from \eqref{eq:radratdyn}. Since the sequence is assumed to be D-algebraic, no further assumptions are required to avoid zero divisors. This is because the defining difference polynomial inherently accounts for the specific indices of the subsequences, thereby ensuring the relevant denominators in the systems are nonzero. 
 
\begin{theorem}\label{Th:subseq} Let $d$ be a positive integer. If $(s(n))_{n\in\NN}$ is D-algebraic then so is $(s(dn))_{n\in\NN}$.
\end{theorem}
\begin{proof} For cleaner formulation, we will use the following iterative notation:

For $N$ variables $x_1,\ldots,x_N$ and a function $f$ in $N$ variables, we write
\begin{align*}
    &x_1,x_2,\ldots,x_N = \underline{x_i}_i,\\
    &f(f(x_1),f(x_2),\ldots,f(x_N)) = f\left(\underline{f(x_i)}_i\right),\\
    &f(f(x_2),f(x_3),\ldots,f(x_N)) = f\left(\underline{f(x_{i\neq 1})}_i\right) = f\left(\underline{f(x_{i>1})}_i\right).
\end{align*}

Let $\tilde{\sigma}=\sigma^d$ and $p\in\mathcal{D}_{\sigma}(x,x)$ of order $r$ such that 
    \begin{equation*}
        p(s(n))=0,\, \text{ for all } n\in\NN.
    \end{equation*}
Define
\begin{equation}
    y_1 = x,\, y_2 = \sigma(x),\ldots, y_{d} = \sigma^{d-1}(x).
\end{equation}
Let $r_1,r_0$ be the quotient and remainder of the Euclidean division of $r+1$ by $d$: $r+1=d\,r_1+r_0$. We can rewrite
\begin{equation*}
    p(x,\sigma(x),\ldots,\sigma^r(x)),
\end{equation*}
as
\begin{equation}
    p\Big(y_1,y_2,\ldots,y_{d},\tilde{\sigma}(y_1),\tilde{\sigma}(y_2),
    \ldots,\tilde{\sigma}^{r_1}(y_1),\tilde{\sigma}^{r_1}(y_2),\ldots,\tilde{\sigma}^{r_1}(y_{r_0})\Big).
\end{equation}
The latter is a multivariate difference polynomial in 
\[\mathcal{D}_{\tilde{\sigma}}(\KK,\{y_1,y_2,\ldots,y_{d}\}).\]
Our aim is to eliminate the difference indeterminates $y_j$, $j>1$. To do so, we build a radical-rational dynamical system and apply \Cref{algo:Algo1}. This will yield a univariate difference polynomial in $y_1$ that has $(s(dn))_{n\in\NN}$ as a zero. Let $m$ be the degree of $\sigma^r(x)$ in $p$ and $R_p$ be the rational expression obtained by solving $p=0$ for $(\sigma^r(x))^m$. Let $y_{i,j}, i=1,\ldots,d$, $j=0,\ldots,r_1$ be new indeterminates such that 
\begin{align}\label{eq:syssubsec}
&y_{i,0}=y_i, i=1,\ldots,d,\\
&\tilde{\sigma}(y_{i,j})=y_{i,j+1},\, i=1,\ldots,d,\, j=0,\ldots,r_1-2, \label{eq:syssub1}\\
&\tilde{\sigma}(y_{i,r_1-1})=y_{i,r_1},\, i=1,\ldots,r_0-1,\label{eq:syssub2}\\
&\tilde{\sigma}(y_{r_0,r_1-1})^m=R_p\left(\underline{y_{i,j}}_{i,j},\tilde{\sigma}(y_{r_0,r_1-1})\right), \label{eq:syssub3}\\
\begin{split}\label{eq:syssub4}
&\tilde{\sigma}(y_{r_0+k,r_1-1})^m=R_p\left(\underline{y_{i\geq k,0}}_i,\underline{y_{i,j\neq 0}}_{i,j},\underline{\tilde{\sigma}(y_{r_0\leq i \leq r_0+k,r_1-1})}_i\right),\\
&k=1,\ldots,d-r_0 
\end{split}
\end{align}
and 
\begin{equation}\label{eq:syssubfin}
\begin{split}
    &\tilde{\sigma}(y_{k,r_1})^m=R_p\Big(\underline{y_{i\geq d-r_0+k,0}}_i,\underline{y_{i,j\neq 0}}_{i,j},
    \underline{\tilde{\sigma}(y_{r_0\leq i \leq r_0+k,r_1})}_i,\underline{\tilde{\sigma}(y_{i\leq k,r_1})}_i\Big),\\
    & k=1,\ldots,r_0-1.
\end{split}
\end{equation}

The equations in \eqref{eq:syssub1}--\eqref{eq:syssubfin} define the desired dynamical system together with the output equation $z=y_{1,0}$.
\end{proof}

\begin{proposition}\label{cor:bound} Let $d$ be a positive integer. If $(s(n))_{n\in\NN}$ is D-algebraic of order $r$ in $\mathcal{D}_{\sigma}(\KK,x)$, then $(s(dn))_{n\in\NN}$ is D-algebraic of order at most $r$ in $\mathcal{D}_{\sigma^d}(\KK,x)$. In other words, $(s(dn))_{n\in\NN}$ is D-algebraic of order at most $dr$ in $\mathcal{D}_{\sigma}(\KK,x)$.
\end{proposition}
\begin{proof} The proof reduces to determining the dimension of the dynamical system obtained in the proof of \Cref{Th:subseq}. Indeed, we have
\begin{align}
&(r_1-1)d\, \text{ equations from }\,  \eqref{eq:syssub1},\\
&r_0-1\, \text{ equations from }\,  \eqref{eq:syssub2},\\
&1\, \text{ equation from }\,  \eqref{eq:syssub3},\\
&d-r_0\, \text{ equations from }\,  \eqref{eq:syssub4},\\
&r_0-1\, \text{ equations from }\,  \eqref{eq:syssubfin},
\end{align}
which give a total of
\begin{equation}
dr_1 + r_0 - 1 = r+1-1=r,
\end{equation}
representing the dimension of the dynamical system and the order of $(s(n))_n$.
\end{proof}

\begin{example}[Illustrative example: $r=4, d=3$] We have $r_1=r_0=1$, $\tilde{\sigma}=\sigma^3$ and $p=p(x,\sigma(x),\ldots,\sigma^4(x))$. We build our dynamical system with the difference indeterminates $y_{1,0}, y_{2,0},y_{3,0}$ and $y_{1,1}$. This yields
\begin{equation}
    \begin{cases}
        \tilde{\sigma}(y_{1,0}) = y_{1,1}\\
        \tilde{\sigma}(y_{2,0})^m = R_p(y_{1,0},y_{2,0},y_{3,0},y_{1,1},\tilde{\sigma}(y_{2,0}))\\
        \tilde{\sigma}(y_{3,0})^m = R_p(y_{2,0},y_{3,0},y_{1,1},\tilde{\sigma}(y_{2,0}),\tilde{\sigma}(y_{3,0}))\\
        \tilde{\sigma}(y_{1,1})^m = R_p(y_{3,0},y_{1,1},\tilde{\sigma}(y_{2,0}),\tilde{\sigma}(y_{3,0}),\tilde{\sigma}(y_{1,1}))\\
        z = y_{1,0}
    \end{cases}.
\end{equation} \QEDA
\end{example}

\begin{remark} The recurrence equation of the subsequence need not be written in the same difference ring as the equation of the original sequence. In \Cref{Th:subseq}, both $(s(n))_{n\in\NN}$ and $(s(dn))_{n\in\NN}$ are zeros of the constructed difference polynomial of order $dr$ in $\mathcal{D}_{\sigma}(\KK,x)$, but $(s(n))_n$ is generally not a zero of the corresponding difference polynomial of order $r$ in $\mathcal{D}_{\sigma^d}(\KK,x)$. This is illustrated in the next example.
\end{remark}

\begin{example}[Explicit example: Catalan numbers at $(3n)_{n\in\NN}$] In \Cref{subsec:holorec}, \Cref{ex:catholotorec}, we used the implementation from \cite{teguia2024rational} to convert the holonomic equation of Catalan numbers $(C(n))_{n\in\NN}$ to the recursion
\begin{equation}
    s(n+2)=\frac{2 s(n+1) \left(8 s(n)+s(n+1)\right)}{10 s(n)-s(n+1)}=R_p(s(n),s(n+1)),
\end{equation}
with $p=(10x-\sigma(x))\sigma^2(x)-2\sigma(x)(8x+\sigma(x))$. Thus, the corresponding dynamical system in $\mathcal{D}_{\tilde{\sigma}}(\KK,x)$, with $\tilde{\sigma}=\sigma^3$, is given by
\begin{equation}
    \begin{cases}
        \tilde{\sigma}(y_{1,0}) = \frac{2 (16 y_{1,0}-y_{2,0}) (8 y_{1,0}+y_{2,0}) y_{2,0}}{(7 y_{1,0}-y_{2,0}) (10 y_{1,0}-y_{2,0})}\\[3mm]
        \tilde{\sigma}(y_{2,0}) = \frac{4 (16 y_{1,0}-y_{2,0}) (8 y_{1,0}+y_{2,0}) y_{2,0} (8 y_{1,0}-y_{2,0})}{(6 y_{1,0}-y_{2,0}) (7 y_{1,0}-y_{2,0}) (10 y_{1,0}-y_{2,0})}\\
        z = y_{1,0}
    \end{cases}.
\end{equation}
We obtain the equation
\begin{multline}\label{eq:catexp}
    343597383680 s(n)^{3} s(n+1)^{3}-69004689408 s(n)^{3} s(n+1)^{2} s(n+2)+4274823168 s(n)^{3} s(n+1) s(n+2)^{2}\\
    -83243160 s(n)^{3} s(n+2)^{3}-1258291200 s(n)^{2} s(n+1)^{4}+266514432 s(n)^{2} s(n+1)^{3} s(n+2)\\
    -26883000 s(n)^{2} s(n+1)^{2} s(n+2)^{2}+1043658 s(n)^{2} s(n+1) s(n+2)^{3}-122880 s(n) s(n+1)^{5}\\
    -101544 s(n) s(n+1)^{4} s(n+2)+65067 s(n) s(n+1)^{3} s(n+2)^{2}-4113 s(n) s(n+1)^{2} s(n+2)^{3}\\
    +1400 s(n+1)^{6}-30 s(n+1)^{5} s(n+2)-75 s(n+1)^{4} s(n+2)^{2}+5 s(n+1)^{3} s(n+2)^{3}=0.
\end{multline} 
One verifies that $(C(n))_n$ is not a solution of \eqref{eq:catexp}. Note that higher-order equations may be obtained here. Indeed, $(C(3n))_n$ is holonomic of order $1$ and degree $3$ and therefore satisfies an l.h.s. algebraic difference equation of order at most $4$ by \Cref{th:holoisratrec}. Our Gr\"obner bases method from \cite{teguia2024rational} yields an equation of order $3$. \QEDA
\end{example}

 This result suggests a potential bridge between difference algebra and classical acceleration techniques. As shown in \cite{brezinski1983convergence}, the careful selection of subsequences can effectively accelerate convergence. For D-algebraic sequences, this indicates the possibility of developing acceleration techniques based entirely on algebraic computations.

\section*{Acknowledgments} The author thanks Antonio Jimenez-Pastor and Veronika Pillewein for their insightful explanations of the resultant method for the D-algebraicity of $D^n$-finite functions, which resulted from stimulating discussions at ACA2025 (\url{https://aca2025.github.io/}) in Heraklion, Greece. The author is also grateful to Filip Mazowiecki and James Worrell for stimulating discussions regarding \cite{clemente2023rational}. This work was supported by UKRI Frontier Research Grant EP/X033813/1.
\bibliographystyle{alpha}

\bigskip

\footnotesize
\noindent {\bf Author's address:}

\smallskip

\noindent  Bertrand Teguia Tabuguia, University of Oxford
\hfill \url{bertrand.teguia@cs.ox.ac.uk}

\newcommand{\etalchar}[1]{$^{#1}$}

\end{document}